\RequirePackage{fix-cm}
\documentclass[smallcondensed,numbook,envcountsame]{svjour3}     

\usepackage{graphicx}

\usepackage{amsmath,epsfig}
\usepackage{cases} 
\usepackage{amssymb,cite}
\usepackage{epsfig}
\usepackage{color}
\usepackage{slashbox}
\usepackage{bm}
\usepackage{cite}
\usepackage{multirow}
\usepackage{slashbox}
\usepackage[utf8]{inputenc}

\newcommand{\Cone}{{\mbox{Cone}}}
\newcommand{\Conv}{{\mbox{Conv}}}

\newcommand{\I}{{\cal I}}
\newcommand{\J}{{\cal J}}
\newcommand{\K}{{\cal K}}
\newcommand{\M}{{\cal M}}

\newcommand{\D}{{\cal D}}
\newcommand{\X}{{\cal X}}
\renewcommand{\S}{\cal S}

\def\bl#1{\textcolor{blue}{#1}}
\begin{document}

\title{
A Smoothing SQP Framework for a Class of Composite $L_q$ Minimization over Polyhedron
}

\titlerunning{An SSQP Framework for a Class of Composite $L_q$ Minimization over Polyhedron}        

\author{Ya-Feng Liu \and Shiqian Ma \and Yu-Hong~Dai \and Shuzhong Zhang}


\institute{Y.-F. Liu and Y.-H. Dai \at
              LSEC, ICMSEC, Academy of
Mathematics and Systems Science, \\
Chinese Academy of Sciences, Beijing 100190, China \\
              \email{yafliu@lsec.cc.ac.cn; dyh@lsec.cc.ac.cn}
           \and
           S. Ma \at
              Department of Systems Engineering and Engineering Management,\\
              The Chinese University of Hong Kong, Shatin, New Territories, Hong Kong, China\\
              \email{sqma@se.cuhk.edu.hk}
              \and
           S. Zhang \at
              Department of Industrial and Systems Engineering,\\
              University of Minnesota, Minneapolis, MN 55455, USA\\
              \email{zhangs@umn.edu}
}

\date{Received: date / Accepted: date}

\maketitle

\begin{abstract}
The composite $L_q~(0<q<1)$ minimization problem over a general polyhedron has received various applications in machine learning, wireless communications, image restoration, signal reconstruction, {etc.} This paper aims to provide a theoretical study on this problem. Firstly, we show that for any fixed $0<q<1$, finding the global minimizer of the problem, even its unconstrained counterpart, is strongly NP-hard. Secondly, we derive Karush-Kuhn-Tucker (KKT) optimality conditions for local minimizers of the problem. Thirdly, we propose a smoothing sequential quadratic programming framework for solving this problem. The framework requires a (approximate) solution of a convex quadratic program at each iteration. Finally, we analyze the worst-case iteration complexity of the framework for returning an $\epsilon$-KKT point; {i.e.}, a feasible point that satisfies a perturbed version of the derived KKT optimality conditions. To the best of our knowledge, the proposed framework is the first one with a worst-case iteration complexity guarantee for solving composite $L_q$ minimization over a general polyhedron.


\keywords{composite $L_q$ minimization \and $\epsilon$-KKT point \and  nonsmooth nonconvex non-Lipschitzian optimization \and  optimality condition \and smoothing approximation \and worst-case iteration complexity }
 \subclass{90C26 \and 90C30 \and 90C46 \and 65K05}
\end{abstract}


\section{Introduction}\label{intro}
In this paper, we consider the following polyhedral constrained composite nonsmooth nonconvex non-Lipschitzian $L_q$ minimization problem
\begin{equation}\label{pnorm}
\begin{array}{cl}
\displaystyle \min_{x} & \displaystyle F(x) := f(x)+h(x)\\
[5pt] \mbox{s.t.} & \displaystyle x\in \X,
\end{array}
\end{equation}
where
\begin{description}[3]
  \item {$f(x)$ is of the form} \begin{equation}\label{f}
    f(x)=\left\|\max\left\{b-Ax,0\right\}\right\|_q^q  = \sum_{m=1}^M \max\left\{b_m-a_m^Tx,0\right\}^q  \end{equation}
      with $A=\left[a_1,a_2,...,a_M\right]^T\in \mathbb{R}^{M\times N},~b=\left[b_1,b_2,...,b_M\right]^T\in\mathbb{R}^{M},$ and $0< q< 1;$
  \item $h(x)$ is a continuously differentiable function with $L_h$-Lipschitz-continuous gradient in $\X,$ that is,  \begin{equation}\label{L}\left\|\nabla h(x)-\nabla h(y)\right\|_2\leq L_h \left\|x-y\right\|_2,~\forall~x,y\in\X;\end{equation}
  \item and, $\X\subseteq\mathbb{R}^N$ is a polyhedral set.
\end{description}

Problem \eqref{pnorm} finds wide applications in information theory\cite{candes05titdecoding}, computational biology\cite{wagner04mpfolding}, wireless communications\cite{sidiropoulos11twcjpac,liu13tspjpac}, machine learning \cite{vapnik92svm,vapnik95svm}, image restoration \cite{nikolova08siamimagenonconvex,chen12tipimage,chen10siamissmoothcg,bianchen14mpfeasible}, signal processing \cite{donoho09siamreviewsparse,mourad10tspnonconvex}, and variable selection\cite{fan01jasavariable,huang09variableselection}. {Three specific applications} arising from machine learning, wireless communications, and information theory are given in Appendix A.

\subsection{Related Works}

Recently, many algorithms have been proposed to solve problem 
\begin{equation}\label{bianchen}\min_{x} h(x)+\left\|x\right\|_q^q.\end{equation}
In particular, when $h(x)$ is a convex quadratic function and $q=1$, problem \eqref{bianchen} is shown to be quite effective in finding a sparse vector to minimize $h(x)$ and various efficient algorithms \cite{beck09siamfista,hale10jcmcontinuation,toh11coapl1,wright09tspsparse,berg08siampareto,nesterov05mpsmooth} have been proposed to solve it.

When $q\in(0,1),$ problem \eqref{bianchen} is nonsmooth, nonconvex, and even not Lipschitz.
{Assuming that $h: \mathbb{R}^N\rightarrow [0,+\infty)$ is continuously differentiable and its gradient} satisfies \eqref{L}, Bian and Chen \cite{bian13siamoptquadratic} proposed a smoothing quadratic regularization (SQR) algorithm for problem \eqref{bianchen}
~and {established the worst-case iteration complexity result, which is $O(\epsilon^{-2}),$ for} the SQR algorithm to return an $\epsilon$-KKT point (or $\epsilon$-KKT solution, or $\epsilon$-stationary point, or $\epsilon$-scaled stationary point, or $\epsilon$-scaled first order stationary point) of problem \eqref{bianchen}. In \cite{bian14mpinterior}, Bian, Chen, and Ye proposed {a first order} interior-point method (using only the gradient information) and a second order interior-point method (using both the gradient and Hessian information) for problem \eqref{bianchen} with box constraints. They showed that the iteration complexity of {their first order method} for returning an $\epsilon$-scaled stationary point is $O(\epsilon^{-2})$ and the one of their second order method for returning an $\epsilon$-scaled second order stationary point is $O(\epsilon^{-3/2}).$
In \cite{chen13siamopttrust}, Chen, Niu, and Yuan derived affine-scaled second order necessary and sufficient conditions for local minimizers of the problem
\begin{equation}\label{chenniu}\min_{x} h(x)+\sum_{m=1}^M|a_m^Tx|^q,\end{equation}which includes \eqref{bianchen} as a special case.
Furthermore, they proposed a smoothing trust region Newton (STRN) method and proved that the sequence generated by the STRN algorithm {is globally convergent} to a point satisfying the affine-scaled second order necessary optimality condition. In \cite{bianchen14mpfeasible}, Bian and Chen proposed {an} SQR {algorithm} for problem \eqref{chenniu} (possibly with box constraints) and showed that the worst-case iteration complexity of the SQR algorithm for finding an $\epsilon$-stationary point is $O(\epsilon^{-2}).$ Cartis, Gould, and Toint \cite{cartis11siamoptcomposite} considered problem
$$\min_{x} h(x)+\varphi(c(x)),$$ where $h: \mathbb{R}^N\rightarrow \mathbb{R}$ and $c: \mathbb{R}^N\rightarrow \mathbb{R}^M$ are continuously differentiable and $\varphi: \mathbb{R}^M\rightarrow \mathbb{R}$ is convex and {is} globally Lipschitz continuous but possibly nonsmooth. They proved that it takes at most $O(\epsilon^{-2})$ iterations to obtain an $\epsilon$-KKT point by a first order trust region method or a quadratic regularization method.
Ghadimi and Lan \cite{lan13} generalized Nesterov's accelerated gradient (AG) method \cite{nesterov83}, originally designed for smooth convex optimization, to solve problem $$\min_{x} h(x)+\psi(x),$$ where $h: \mathbb{R}^N\rightarrow \mathbb{R}$ is continuous differentiable and $\psi(x): \mathbb{R}^N\rightarrow \mathbb{R}$ is a (simple) convex nonsmooth function with special structures. They showed that it takes at most $O(\epsilon^{-2})$ iterations to reduce a first order criticality measure below $\epsilon$ for the generalized AG method. Jiang and Zhang \cite{jiangzhang14} considered the following block-structured problem
\begin{equation*}
\begin{array}{cl}
\displaystyle \min_{x} & \displaystyle h(x_1,x_2,\ldots,x_M)+\sum_{m=1}^M \psi_m(x_m)\\
[5pt] \mbox{s.t.} & \displaystyle x_m\in \X_m, m=1,2,\ldots,M,
\end{array}
\end{equation*}where $h: \mathbb{R}^N\rightarrow \mathbb{R}$ is smooth, $\psi_m(x_m): \mathbb{R}^{N_m}\rightarrow \mathbb{R}$ are convex but nonsmooth. They showed that the conditional gradient and gradient projection type of methods can find an $\epsilon-$KKT point of the above problem within $O\left(\epsilon^{-2}\right)$ iterations. {Here we should notice} that the definitions of $\epsilon-$KKT points in the aforementioned works are different and thus are not comparable to each other.

In particular, when $h(x)=\frac{\rho}{2}\left\|Ax-b\right\|^2,$ problem \eqref{bianchen} becomes
\begin{equation}\label{l2lp}\min_{x}\frac{\rho}{2}\|Ax-b\|^2+\|x\|_q^q.\end{equation}
Chen et al. \cite{cheng14mpl2lpcomplexity} showed that problem \eqref{l2lp} is strongly NP-hard. 
Recently, iterative reweighted $L_1$ and $L_2$ minimization algorithms are proposed to (approximately) solve problem \eqref{l2lp} (see \cite{chen10siamissmoothcg,daubechies10iterative,lai09,lai11siamoptlq,lai13siamimproved} and {the} references therein). In \cite{Chen10siamoptlowerbound}, Chen, Xu, and Ye derived a lower bound theory for local minimizers of problem \eqref{l2lp}; i.e., each component of any local minimizer of problem \eqref{l2lp} is either zero or not less than a positive constant which only depends on the problem inputs $A,b,\rho,$ and $q.$ Lu \cite{Lu12iterative} extended the lower bound theory to problem \eqref{bianchen} with general $h(x)$ satisfying \eqref{L}. Based on the derived lower bound theory, Lu proposed a {novel} iterative reweighted minimization method for solving problem \eqref{bianchen} and provided a unified global convergence analysis for the aforementioned iterative reweighted minimization algorithms.

Another problem closely related to problem \eqref{l2lp} is
\begin{equation}\label{ye}
\displaystyle \min_{x}  \displaystyle \left\|x\right\|_q^q, \ \mbox{s.t.} \quad \displaystyle Ax=b.
\end{equation}
In \cite{ge11mpnote}, Ge, Jiang, and Ye showed that problem \eqref{ye} and its {smoothed} version are strongly NP-hard. Chartrand \cite{chartrand07splexact}, Chartrand and Staneva \cite{chartrand08iprip}, Foucart and Lai \cite{lai09}, and Sun \cite{sun12acharecovery} established some sufficient conditions under which problem \eqref{ye} is able to recover the {sparsest} solution to the undetermined linear system $Ax=b.$
Efficient iterative reweighted minimization algorithms were proposed to solve problem \eqref{ye} by Chartrand and Yin \cite{chartrand08icassp}, Foucart and Lai \cite{lai09}, Daubechies et al. \cite{daubechies10iterative}, Rao and Kreutz-Delgado \cite{rao99tspaffine}, and Cand\`es, Wakin, and Boyd  \cite{candes08reweight}. It was shown in \cite[Theorem 7.7(i)]{daubechies10iterative} that under suitable conditions, the sequence generated by the iterative reweighted $L_2$ minimization algorithms converges to the global minimizer of problem \eqref{ye}. Moreover, the following related problem
\begin{equation}\label{ye-nonnegative}
\displaystyle \min_{x} \displaystyle \left\|x\right\|_q^q, \ \mbox{s.t.} \quad \displaystyle Ax=b,\ x\geq 0,
\end{equation}
was also considered in \cite{ge11mpnote}, and the authors developed an interior-point potential reduction algorithm for solving problem \eqref{ye}, which is guaranteed to return a scaled $\epsilon$-KKT point in no more than $O(\epsilon^{-1}\log\epsilon^{-1})$ iterations. The similar idea was extended by Ji et al. \cite{anthony13infocom} to solve the matrix counterpart of problem \eqref{ye-nonnegative} where the unknown variable is a positive semidefinite matrix.

{Although many algorithms have been mentioned in the above, they cannot be used to solve problem \eqref{pnorm}.} For instance, the potential reduction algorithm in \cite{ge11mpnote} cannot be applied to solve problem \eqref{pnorm} where $h(x)$ is not concave; the SQR algorithms \cite{bian13siamoptquadratic,bianchen14mpfeasible} and the interior-point algorithms \cite{cheng14mpl2lpcomplexity} cannot deal with the composite $L_q$ term $f(x)$ and the general polyhedral constraint in problem \eqref{pnorm}; the algorithm proposed in \cite{cartis11siamoptcomposite} cannot be used to solve problem \eqref{pnorm} either, since the composite $L_q$ term $f(x)$ in the objective function of \eqref{pnorm} cannot be expressed as a form of $\varphi(c(x)).$ The aforementioned iterative reweighted minimization methods could be modified to solve problem \eqref{pnorm}. However, the worst-case iteration complexity of all existing iterative reweighted minimization methods remains unclear so far and global convergence of some of them are still unknown \cite{candes08reweight}. The goal of this paper is to develop an algorithmic framework for problem \eqref{pnorm} with worst-case iteration complexity guarantee.
%
\subsection{Our Contribution}
In this paper, we consider {polyhedral constrained} composite nonsmooth nonconvex non-Lipschitzian $L_q$ minimization problem \eqref{pnorm}, which includes problems \eqref{bianchen}, \eqref{chenniu}, \eqref{l2lp}, \eqref{ye}, and \eqref{ye-nonnegative} as special cases. A sharp difference between problem \eqref{pnorm} and {the} aforementioned problems lies in the composite term in $\|\cdot\|_q^q,$ i.e., problem \eqref{pnorm} tries to find a solution such that the number of positive components of the vector $b-Ax$ is as small as possible. However, problem \eqref{bianchen}, for instance, tries to find a solution such that the number of nonzero entries of $x$ is as small as possible. In other words, problem \eqref{pnorm} considered in this paper is essentially a sparse optimization problem with inequality constraints while all previously mentioned problems are sparse optimization with equality constraints.

We propose a {smoothing sequential quadratic programming (SSQP) framework} for solving problem \eqref{pnorm}, where a convex quadratic program (QP) is (approximately) solved at each iteration, and analyze the worst-case iteration complexity of the proposed algorithmic framework. One iteration in this paper refers to (approximately) solving one convex QP subproblem. To the best of our knowledge, this is the first algorithm/framework for solving polyhedral constrained composite non-Lipschitzian $L_q$ minimization with worst-case iteration complexity analysis. The {main} contributions of this paper are summarized as follows.
\begin{description}
  \item [-] {Problem \eqref{pnorm} is shown to be strongly NP-hard even when $h(x)=0$ and $\X=\mathbb{R}^N$ (see Theorem \ref{hardness});}
  \item [-] {KKT optimality conditions for local minimizers of problem \eqref{pnorm} are derived (see Theorem \ref{thm-optcondition});}
  \item [-] {A lower bound theory is developed for local minimizers of problem \eqref{pnorm} assuming that $h(x)$ is concave (see Theorem \ref{thm-lowerbound});}
  \item [-] {An SSQP algorithmic framework is proposed for solving problem \eqref{pnorm} and its worst-case iteration complexity is analyzed. In particular, we show in Theorem \ref{thm-complexity2} that the SSQP framework can return an $\epsilon$-KKT point of problem \eqref{pnorm} in Definition \ref{definitione} within $O\left(\epsilon^{q-4}\right)$ iterations. Here we should notice that the $\epsilon$-KKT point defined in Definition \ref{definitione} is stronger than the ones used in \cite{bian13siamoptquadratic,Chen10siamoptlowerbound,ge11mpnote,chen13siamopttrust,bianchen14mpfeasible} when problem \eqref{pnorm} reduces to problems \eqref{bianchen} and \eqref{chenniu}.}
\end{description}

%
%
%

The rest of this paper is organized as follows. In Section \ref{sec-complexity}, we show that problem \eqref{pnorm} (with $h(x)=0$ and $\X=\mathbb{R}^N$) is strongly NP-hard. In Section \ref{sec-optcondition}, we show that problem \eqref{pnorm} and auxiliary problem \eqref{psub4} are equivalent in the sense that the two problems share {the} same local minimizers. This equivalence {result} further implies the KKT optimality conditions as well as the lower bound theory for local minimizers of problem \eqref{pnorm} under the assumption that $h(x)$ is concave. In Section \ref{sec-smoothapp}, we give a smoothing approximation for problem \eqref{pnorm}. In Section \ref{sec-algorithm}, we propose an SSQP algorithmic framework for solving problem \eqref{pnorm} and give the worst-case iteration complexity of the proposed algorithmic framework. Finally, we make some concluding remarks in Section \ref{sec-conclusion}.



\textbf{Notations.} 
{We always denote $\M=\left\{1,2,\ldots,M\right\}$. For any set $\K,$ $|\K|$ stands for its cardinality. $\nabla h(x)$ is the gradient of a continuously differentiable function $h(x)$. $P_{\X}(x)$ is the projection of a point $x$ onto the convex set $\X.$  $I_N$ is the $N\times N$ identity matrix. Throughout this paper, $\|\cdot\|$ denotes the Euclidean norm unless otherwise specified.}

\section{Intractability Analysis}\label{sec-complexity}


In this section, we show that the $L_q$ minimization problem \eqref{pnorm} with any $q\in(0,1)$ is strongly NP-hard by proving the strong NP-hardness of its special case 
\begin{equation}\label{unconstrained}
  \min_{x}\left\|\max\left\{b-Ax,0\right\}\right\|_q^q.
\end{equation} The proof is based on a polynomial time
transformation from the strongly NP-complete 3-partition~problem\cite[Theorem 4.4]{garey79complexity}. The 3-partition problem can be described as follows: given a set of positive integers $\left\{a_i\right\}_{i\in\cal S}$ with ${\cal S}=\left\{1,2,\ldots,3m\right\}$ and a positive integer $B$ such that
$a_i\in(B/4,B/2)$~for all $i\in\cal S$~and~\begin{equation}\label{sum}\sum_{i\in\cal S}a_i=mB,\end{equation} the problem is to check whether there exists a partition ${\cal S}_{1},\ldots,{\cal S}_{m}$ of $\cal S$ such that \begin{equation}\label{Si}\sum_{i\in{\cal S}_j}a_i=B,~\forall~j=1,2,\ldots,m.\end{equation}
Notice that the constraints on $\left\{a_i\right\}$ imply that each ${\cal S}_j$ in \eqref{Si} must contain exactly three elements. 

\begin{theorem}\label{hardness}
{For any $q\in (0, 1)$, the unconstrained $L_q$ minimization problem \eqref{unconstrained} is strongly NP-hard and hence so is the polyhedral constrained $L_q$ minimization problem \eqref{pnorm}.}
\end{theorem}

\begin{proof}
We prove {the theorem} by constructing a polynomial time transformation from the 3-partition problem to the unconstrained $L_q$ minimization problem \eqref{unconstrained}\cite{garey79complexity,papadimitriou94complexity,vazirani01}.
For any given instance of the 3-partition problem with $\left\{a_i\right\}_{i\in\cal S},$ $m,$ and $B,$ we construct an instance of problem \eqref{unconstrained} with $M=3m^2+4m$ and $N=3m^2$ as follows:
\begin{equation}\label{instance-unc}
  \min_{x} f(x):=f_1(x)+f_2(x)+f_{3}(x),
\end{equation}
where \begin{eqnarray*}
  f_1(x)&=&\sum_{i=1}^{3m}\sum_{j=1}^m\left(\max\left\{x_{ij},0\right\}^q+\max\left\{1-x_{ij},0\right\}^q\right),\label{f1}\\
  \displaystyle f_2(x)&=&\sum_{i=1}^{3m} \max\left\{\sum_{j=1}^mx_{ij}-1,0\right\}^q,\label{f2}\\
  f_3(x)&=&\sum_{j=1}^m\max\left\{B-\sum_{i=1}^{3m}a_ix_{ij},0\right\}^q.\label{f3}
\end{eqnarray*}
It is easy to verify that $f_1(x)\geq 3m^2,~f_2(x)\geq 0,~f_3(x)\geq 0.$ Moreover, $f(x)=3m^2$ if and only if
\begin{eqnarray}
  &&x_{ij}\in\left\{0,\,1\right\},~\forall~i=1,2,\ldots,3m,~\forall~j=1,2,\ldots,m,\label{f1=}\\
  &&\sum_{j=1}^mx_{ij}\leq 1,~\forall~i=1,2,\ldots,3m,\label{f2=}\\
  &&\sum_{i=1}^{3m}a_ix_{ij}\geq B,~\forall~j=1,2,\ldots,m.\label{f3=}
\end{eqnarray}

Next, we show that the global minimum of problem \eqref{instance-unc} is not greater than $3m^2$ if and only if the answer to the 3-partition problem is yes.
We divide this into two steps. (a) ``if'' direction. Assuming that there exists a partition of $\S$ such that \eqref{sum} holds true, the system \eqref{f1=}, \eqref{f2=}, and \eqref{f3=} (with inequalities in \eqref{f2=} and \eqref{f3=} replaced by equalities) must have a feasible solution $x,$ which further implies $f(x)=3m^2.$ Thus the optimal value of problem \eqref{instance-unc} is not greater than $3m^2.$ (b) ``only if'' direction.
Assuming that there exists a point $x$ such that $f(x)=3m^2,$ we know that \eqref{f1=}, \eqref{f2=}, and \eqref{f3=} hold true at $x.$
In this case, by \eqref{f3=}, \eqref{f2=} and \eqref{sum}, we can get that
$$mB\leq \sum_{j=1}^{m}\sum_{i=1}^{3m}a_ix_{ij}=\sum_{i=1}^{3m}a_i\sum_{j=1}^{m}x_{ij}\leq \sum_{i=1}^{3m}a_i=mB.$$
Thus \eqref{f2=} and \eqref{f3=} must hold with equalities. Combining this with \eqref{f1=}, \eqref{f2=}, and \eqref{f3=}, we can see that
$x$ corresponds to a partition of $\S.$

Finally, since this transformation can be {done} in polynomial time and the 3-partition problem is strongly NP-complete, we conclude that problem \eqref{unconstrained} and thus problem \eqref{pnorm} are strongly NP-hard. \qed
\end{proof}
{Theorem \ref{hardness} indicates that, for any $q\in(0,1)$, it is computationally intractable to find the global minimizer of problem \eqref{pnorm} or even its unconstrained counterpart \eqref{unconstrained}.}  

\section{KKT Optimality {Conditions}}\label{sec-optcondition}
In this section, we derive the KKT optimality conditions and {a} lower bound theory for local minimizers of problem \eqref{pnorm}.
~{To do so, we introduce} an auxiliary problem \eqref{psub4} and {establish} a key one-to-one correspondence of local minimizers of problems \eqref{pnorm} and \eqref{psub4}.

For any given $\bar x\in\X,$ define {the sets}
\begin{equation}\label{M1M2M3}
  \begin{array}{rcl}
  \I_{\bar x}&=&\left\{m\,|\,(b-A\bar x)_m<0\right\},\\
  \J_{\bar x}&=&\left\{m\,|\,(b-A \bar x)_m>0\right\},\\
  \K_{\bar x}&=&\left\{m\,|\,(b-A\bar x)_m=0\right\},\end{array}\end{equation}
and {the corresponding problem}
 \begin{equation}\label{psub4}
\begin{array}{cl}
\displaystyle \min_{x} & \displaystyle \sum_{m\in\J_{\bar x}} (b-Ax)_m^q + h(x) \\
\mbox{s.t.} & (b-Ax)_{m}\leq 0,~m\in{\K}_{\bar x},\\
& \displaystyle x\in \X.
\end{array}
\end{equation} Notice that the objective value of problem \eqref{psub4} is equal to {that} of problem \eqref{pnorm} at point $\bar x.$ Moreover, the objective {function of} problem \eqref{psub4} is continuously differentiable in the neighborhood of point $\bar x.$. 


It is easy to verify the following lemma.
\begin{lemma}\label{lemma-easy}
  If $\bar x$ is a local minimizer of problem \eqref{pnorm}, then it is {also} a local minimizer of problem \eqref{psub4} with $\J_{\bar x}$ and $\K_{\bar x}$ given in \eqref{M1M2M3}.
\end{lemma}
The following lemma {indicates} that the converse of Lemma \ref{lemma-easy} is also true.

\begin{lemma}\label{lemma-local}
If $\bar x$ is a local minimizer of problem \eqref{psub4} with $\J_{\bar x}$ and $\K_{\bar x}$ given in \eqref{M1M2M3}, then it is {also} a local minimizer of problem \eqref{pnorm}.
\end{lemma}
{Lemma \ref{lemma-local} can be verified by using the two facts:} the feasible direction cone of problem \eqref{pnorm} at any feasible point is finitely generated (because $\X$ is a polyhedral set) and the function $z^q$ is non-Lipschitz and concave with respect to $z\geq 0.$ Since the detailed proof of Lemma \ref{lemma-local} is technical, we relegate it to Appendix B.



{We are now ready to provide the main theorem of this section, which} presents the KKT optimality conditions for local minimizers of problem \eqref{pnorm}.
\begin{theorem}\label{thm-optcondition}
\textbf{[KKT Optimality Conditions]} 
If $\bar x\in\X$ is a local minimizer of problem \eqref{pnorm}, {there must exist} $\bar\lambda\geq0\in\mathbb{R}^{\left|\K_{\bar x}\right|}$ such that 
\begin{equation}\label{KKT1}\bar\lambda_m(b-A\bar x)_m=0,~\forall~m\in\K_{\bar x}\end{equation} and
\begin{equation}\label{KKT2}\bar x-P_{\X}\left(\bar x-\nabla L(\bar x,\bar\lambda)\right)=0,\end{equation}where
\begin{equation}\label{Lag}\begin{array}{rl}L(x,\lambda)=&\displaystyle \sum_{m\in\J_{\bar x}}(b-Ax)_m^q+h(x)\displaystyle +\sum_{m\in\K_{\bar x}}\lambda_m(b-Ax)_m,\end{array}\end{equation}
and $\J_{\bar x}$ and $\K_{\bar x}$ are defined in \eqref{M1M2M3}.
\end{theorem}
\begin{proof}
  By Lemmas \ref{lemma-easy} and \ref{lemma-local}, $\bar x$ is a local minimizer of problem \eqref{pnorm} if and only if it is a local minimizer of problem \eqref{psub4} with $\J_{\bar x}$ and $\K_{\bar x}$ given in \eqref{M1M2M3}. Combining this equivalence and the fact that $L(x,\lambda)$ in \eqref{Lag} is the Lagrangian function of problem \eqref{psub4} with $\lambda$ being the associated Lagrangian multiplier, we obtain \eqref{KKT1} and \eqref{KKT2} immediately.\qed
\end{proof}

Note that the following version of the KKT point (or stationary point, or scaled KKT point, or scaled stationary point, or first-order stationary point) for problems \eqref{bianchen} and \eqref{chenniu} has been used in many previous works (see, e.g., \cite{bian13siamoptquadratic,Chen10siamoptlowerbound,ge11mpnote,chen13siamopttrust,bianchen14mpfeasible}).
\begin{definition}\label{def:kkt}
  $\bar x$ is called a KKT point of problem \eqref{bianchen} if it satisfies \begin{equation}\label{kktother}q|\bar x|^q+\bar X\nabla h(\bar x)=0,\end{equation} where $|\bar x|^q=\left(|\bar x_1|^q,\ldots,|\bar x_N|^q\right)^T$ and $\bar X=\text{diag}\left(\bar x_1,\ldots,\bar x_N\right).$
\end{definition}

\begin{definition}\label{def:kkt2}
  $\bar x$ is called a KKT point of problem \eqref{chenniu} if it satisfies \begin{equation}\label{kktother2}
  Z_{\bar x}^T\nabla F_{{\bar x}}(\bar x)=0,\end{equation} where
  \begin{equation*}\label{reducedF}F_{{\bar x}}(x)=\sum_{a_m^T{\bar x}\neq 0} \left|a_m^Tx\right|^q+h(x)\end{equation*} 
  and $Z_{\bar x}$ is the matrix whose columns form an orthogonal basis for the null space of $\left\{a_m\,|\,a_m^T\bar x=0\right\}.$ 
\end{definition}

In the following, we show that our definition of the KKT point for problem \eqref{pnorm} in Theorem \ref{thm-optcondition} reduces to the ones in Definitions \ref{def:kkt} and \ref{def:kkt2} when problem \eqref{pnorm} reduces to problems \eqref{bianchen} and \eqref{chenniu}, respectively.

\begin{proposition}
  When problem \eqref{pnorm} reduces to problem {\eqref{chenniu}}, there holds \begin{equation}\label{equivalence-reduce}\eqref{KKT1}~\text{and}~\eqref{KKT2}\Longleftrightarrow \eqref{kktother2};\end{equation}
  When problem \eqref{pnorm} reduces to problem {\eqref{bianchen}}, there holds $$\eqref{KKT1}~\text{and}~\eqref{KKT2}\Longleftrightarrow \eqref{kktother}.$$
\end{proposition}
\begin{proof}
For succinctness, we only show the first statement of the proposition. The second one can be shown by using the same arguments. When problem \eqref{pnorm} reduces to \eqref{chenniu}, problem \eqref{psub4} reduces to
\begin{equation*}\label{psub5}
\begin{array}{cl}
\displaystyle \min_{x} & \displaystyle \sum_{m\in\hat\I_{\bar x}} \left(-a_m^Tx\right)^q+\sum_{m\in\hat\J_{\bar x}} \left(a_m^Tx\right)^q + h(x) \\
\mbox{s.t.} & a_m^Tx=0,~m\in\hat{\K}_{\bar x},
\end{array}
\end{equation*} with 
\begin{equation*}\label{M1M2M3-special}
  \begin{array}{rcl}
  \hat\I_{\bar x}=\left\{m\,|\,a_m^T\bar x<0\right\},~\hat\J_{\bar x}=\left\{m\,|\,a_m^T\bar x>0\right\},~\text{and}~\hat\K_{\bar x}&=&\left\{m\,|\,a_m^T\bar x=0\right\}.\end{array}\end{equation*}
Therefore, the KKT optimality conditions $\eqref{KKT1}~\text{and}~\eqref{KKT2}$ in Theorem \ref{thm-optcondition} reduce to {the following}: there exists $\bar\lambda\in\mathbb{R}^{\left|\hat\K_{\bar x}\right|}$ 
  such that
\begin{equation}\label{kkt1-special}\bar\lambda_{m}a_m^T\bar x=0,~\forall~m\in\hat\K_{\bar x},\end{equation}
and
\begin{equation}\label{kkt2-special}\nabla \hat L(\bar x,\bar\lambda)=0,\end{equation}
where
\begin{equation*}\label{LAb-special}
\begin{array}{rl}
\hat L(x,\lambda)=&\displaystyle \sum_{m\in\hat\I_{\bar x}} \left(-a_m^Tx\right)^q+\sum_{m\in\hat\J_{\bar x}} \left(a_m^Tx\right)^q+h(x)+\sum_{m\in\hat\K_{\bar x}}\lambda_m a_m^Tx.
\end{array}\end{equation*}

 Hence, to show \eqref{equivalence-reduce}, it suffices to show that
  $$\eqref{kkt1-special}~\text{and}~\eqref{kkt2-special}\Longleftrightarrow \eqref{kktother2}.$$
  {To establish} the direction ``$\Longrightarrow$'', we
recall the definitions of $Z_{\bar x}$ (see Definition \ref{def:kkt2}) and $\hat L(x,\lambda).$ {By \eqref{kkt2-special}, we immediately have}
$$Z_{\bar x}^T\nabla F_{{\bar x}}(\bar x)=Z_{\bar x}^T\nabla\hat L(\bar x,\bar\lambda)=0.$$
 {To establish} the direction ``$\Longleftarrow$'', we have by \eqref{kktother2} and the {definition} of $Z_{\bar x}$ that 
\begin{align*}
  \nabla F_{{\bar x}}(\bar x)\in Z_{\bar x}^{\perp}&\Longleftrightarrow \nabla F_{{\bar x}}(\bar x)\in\text{span}\left\{a_m,m\in\hat{\K}_{\bar x}\right\}\\
  &\Longleftrightarrow \nabla F_{{\bar x}}(\bar x)=-\sum_{m\in\hat{\K}_{\bar x}} \bar \lambda_m a_m~\text{for~some~}\left\{\bar \lambda_m\right\}_{m\in\hat{\K}_{\bar x}},
\end{align*} which implies \eqref{kkt2-special} in turn. Due to the definition of $\hat{\K}_{\bar x}$, \eqref{kkt1-special} holds true trivially. The proof is completed. \qed
\end{proof}

Next, we extend the lower bound theory for local minimizers of the unconstrained problem \eqref{bianchen} in \cite{Chen10siamoptlowerbound,Lu12iterative} to the {polyhedral constrained} problem \eqref{pnorm}. Suppose that $h(x)$ in \eqref{pnorm} is concave with respect to $\X$. Then, for any given $\I,~\J,~\K$ {satisfying} $\I\cup\J\cup\K=\M,$ the objective function of problem \eqref{psub4} is concave in 
\begin{equation}\label{polytope}
\left\{ x\in \X\,|\,(b-Ax)_{{\I}\cup\K}\leq 0,~(b-Ax)_{{\J}}\geq 0\right\}.
\end{equation}
Therefore, 
all local minimizers of problem \eqref{psub4} must be vertices of the polytope \eqref{polytope} except that there exists an edge direction $d$ connecting two vertices $x_1$ and $x_2$ such that $(Ad)_{\J}=0.$ In the latter case, any convex combination of $x_1$ and $x_2$ is a local minimizer of problem \eqref{psub4}. 


We have the following lower bound theory for local minimizers of problem \eqref{pnorm}.
\begin{theorem}[Lower Bound Theory]\label{thm-lowerbound}
  Suppose {that} $h(x)$ is concave with respect to $x\in \X$ and $\bar x$ is any local minimizer of problem \eqref{pnorm}. Then, for any $m\in\M,$ we have either $(b-A\bar x)_m\leq 0$ or {$(b-A\bar x)_m\geq C,$} where {$C$} is a positive constant that only depends on $A,~b,$ and $\X.$
\end{theorem}
\begin{proof}
 From Lemma \ref{lemma-easy} and the argument before the theorem, we know that any local minimizer of problem \eqref{pnorm}, $\bar x$, must be a vertex of the polytope \eqref{polytope} with $\I_{\bar x},\,\J_{\bar x},$ and $\K_{\bar x}$ given in \eqref{M1M2M3} and satisfies $(b-A\bar x)_{\J_{\bar x}}>0.$ For any given $\I,~\J,~\K$ satisfying $\I\cup\J\cup\K=\M,$ denote the vertex set of the polytope \eqref{polytope} by $V(A,b,\X,\I,\J,\K)$ and let $$V^{\circ}(A,b,\X,\I,\J,\K):=V(A,b,\X,\I,\J,\K)\bigcap \left\{x\,|\,(b-Ax)_{\J}>0\right\}.$$ Then, the set of local minimizers of problem \eqref{pnorm} must belong to
 $$V(A,b,\X):=\bigcup_{\I\cup\J\cup\K=\M} V^{\circ}(A,b,\X,\I,\J,\K).$$
 Since the polytope \eqref{polytope} has finitely many vertices\cite[Proposition 3.3.3]{bertsekas03convex} and the number of the partition of the set $\M$ is finite, it follows that the set $V(A,b,\X)$ contains finitely many points,
 ~which further implies
$$C(A,b,\X):=\min_{\I\cup\J\cup\K=\M}\min_{x\in V^{\circ}(A,b,\X,\I,\J,\K)}\min\left\{(b-Ax)_{\J}\right\}>0.$$
This shows that the lower bound theory holds true for problem \eqref{pnorm}, i.e., each component of $b-A\bar x$ at any local minimizer $\bar x$ of problem \eqref{pnorm} is either not greater than zero or not less than a constant {$C(A,b,\X).$}  \qed
%
\end{proof}

\section{Smoothing Approximation}\label{sec-smoothapp}

Smoothing approximations for nonsmooth minimization have been extensively studied in \cite{nesterov05mpsmooth,chen12mpsmoothreview,bian13siamoptquadratic,zhang05resursivemax} and {the} references therein. In this section, we propose to use the smooth function 
\begin{equation}\label{thetatmu}\theta(t,\mu)=\left\{\begin{array}{cl}t, &\text{if}~t>\mu;\\
                                                                      \frac{t^2}{2\mu}+\frac{\mu}{2},\quad &\text{if}~0\leq t\leq\mu;\\
                                                                      \frac{\mu}{2}, &\text{if}~t<0 \\
\end{array}\right. \end{equation}
to approximate the max function \begin{equation*}\label{theta}\theta(t)=\max\left\{t,0\right\}.\end{equation*} Based on {\eqref{thetatmu}},
we can construct a smoothing function $\tilde F$ of $F$ and thus a smoothing approximation problem of nonsmooth problem \eqref{pnorm}.

We first summarize some useful properties of $\theta(t,\mu).$ Clearly, for any fixed $\mu>0,$ we have
\begin{equation*}\label{equal=}
\theta(t,\mu)=\theta(t),~\forall~t\geq \mu,
\end{equation*}and
\begin{equation}\label{lowerq}
\theta(t,\mu)\geq\frac{\mu}{2},~\forall~t.
\end{equation}
In addition, $\theta^q(t,\mu)$ is continuously differentiable and twice continuously differentiable everywhere except at the points $t=0$ and $t=\mu.$ The first and second order derivatives of $\theta^q(t,\mu)$ with respect to $t$ are given as follows:
\begin{equation}\label{q1d}\left[\theta^q(t,\mu)\right]'=\left\{\begin{array}{cl}\displaystyle qt^{q-1}, &\text{if}~t>\mu;\\
                                                                      \displaystyle q\theta^{q-1}(t,\mu)\frac{t}{\mu},\quad &\text{if}~0\leq t\leq\mu;\\
                                                                      \displaystyle 0, &\text{if}~t<0,\\
\end{array}\right.\end{equation}

\begin{equation}\label{q2nd}\left[\theta^q(t,\mu)\right]''=\left\{\begin{array}{cl}\displaystyle q(q-1)t^{q-2}, &\text{if}~t>\mu;\\
                                                                      \displaystyle q\left(q-1\right)\theta^{q-2}(t,\mu)\frac{t^2}{\mu^2}+q\theta^{q-1}(t,\mu)\frac{1}{\mu}, \quad &\text{if}~0< t<\mu;\\
                                                                      \displaystyle 0, &\text{if}~t<0.\\
\end{array}\right.\end{equation}

\begin{lemma}\label{lemma-smooth}For any $q\in(0,1)$ and $\mu\in (0,+\infty),$ the following statements {hold true}.
  \begin{description}
    \item [(i)] $0\leq \theta^q(t,\mu)-\theta^q(t)\leq\left(\frac{\mu}{2}\right)^q,~\forall~t\in(-\infty,\mu];$
    \item [(ii)] $\max\left\{|\upsilon|\,|\,\upsilon\in\partial_{t}\left(\left[\theta^q(t,\mu)\right]'\right)\right\}\leq {4q}\mu^{q-2},~\forall~t\in\mathbb{R},$ where
         $\partial_t$ denotes the Clarke generalized gradient with respect to $t$\cite{clarke};
  \item [(iii)]  Define \begin{equation}\label{kappa}
   \displaystyle \kappa(t,\mu)=\left\{
\begin{array}{cl}
{4q}\mu^{q-2},\quad &\mbox{if $-\mu\leq t\leq 2\mu$;}\\
0,&\mbox{otherwise.}\\
\end{array}\right.
  \end{equation} Then
  \begin{equation}\label{q-upperbnd}\theta^q(t,\mu)\leq \theta^q(\hat t,\mu)+\left[\theta^q(\hat t,\mu)\right]'\left(t-\hat t\right)+\frac{\kappa(\hat t,\mu)}{2}\left(t-\hat t\right)^2\end{equation} for any $t$ and $\hat t$ such that $t-\hat t\geq {-\hat t}/{2}$ if $\hat t> 2\mu,$ or $t\in(-\infty,+\infty)$ if $-\mu\leq \hat t\leq  2\mu,$ or $ t-\hat t\leq \mu$ if $\hat t< -\mu.$
  \end{description}
\end{lemma}
\begin{proof}
  See Appendix C.\qed
\end{proof}
%
%
%


Define \begin{equation}\label{tildeF}\tilde F(x,\mu)=\tilde f(x,\mu)+h(x),\end{equation}where\begin{equation}\label{tildef}\displaystyle \tilde f(x,\mu)=\sum_{m\in\M}\theta^q((b-Ax)_m,\mu).\end{equation}
Based on Lemma \ref{lemma-smooth} and the discussions {beforehand}, we know that $\tilde F(x,\mu)$ is a smoothing function of $F(x)$ and satisfies 
\begin{equation}\label{tildeFin}F(x)\leq \tilde F(x,\mu)\leq F(x)+\sum_{(b-Ax)_m\leq \mu}\left(\frac{\mu}{2}\right)^q,~\forall~x,\end{equation}
and
\begin{equation}\label{tildeg}\nabla\tilde F(x,\mu)=\nabla\tilde f(x,\mu)+\nabla h(x)=-\sum_{m\in\M} \left[\theta^q(t,\mu)\right]'_{t=(b-A x)_m} a_m+\nabla h(x).\end{equation}
Therefore,  \begin{equation}\label{smoothingproblem}
\begin{array}{cl}
\displaystyle \min_{x} & \displaystyle \tilde F(x,\mu) \\
\mbox{s.t.} & \displaystyle  x\in \X
\end{array}
\end{equation} is a smoothing approximation to problem \eqref{pnorm}.

Using essentially the same arguments as in the proof of Theorem \ref{hardness}, we can show the following result.
\begin{theorem}
  For any $q\in(0,1)$ and $\mu>0,$ the smoothing approximation problem \eqref{smoothingproblem} is strongly NP-hard ({even for the special case when $h(x)=0$ and $\X=\mathbb{R}^N$}).
\end{theorem}

\section{An SSQP Framework and Worst-Case Iteration Complexity Analysis}\label{sec-algorithm}

In this section, we propose a smoothing SQP (SSQP) algorithmic framework for solving problem \eqref{pnorm}. The proposed algorithmic framework (approximately) solves a convex QP at each iteration. The objective function of the QP subproblem is constructed as a local upper bound of the smoothing function $\tilde F(x,\mu)$ in \eqref{tildeF}. 
In the proposed SSQP framework, the smoothing parameter is updated if the residual of the smoothing problem \eqref{smoothingproblem} is {not greater than some} constant (depending on the current smoothing parameter). We shall also analyze the worst-case iteration complexity of the proposed framework.

{Specifically, we construct} a local convex quadratic upper bound of $\tilde F(x,\mu)$ and present the SSQP algorithmic framework for problem \eqref{pnorm} in Subsection \ref{subsec-algorithm}. Then we define the $\epsilon$-KKT point of problem \eqref{pnorm} and analyze the worst-case iteration complexity of the proposed algorithm/framework for obtaining an $\epsilon$-KKT point in Subsection \ref{subsec-analysis}. Finally, we compare the proposed SSQP algorithm/framework with some existing algorithms {in Subsection 5.3.}

\subsection{An SSQP Algorithmic Framework for Problem \eqref{pnorm}}\label{subsec-algorithm}

For any fixed $\mu>0,$ define the quadratic approximation of $\tilde f(\cdot,\mu)$ around $x_k$ as
\begin{equation}\label{quad-upper}
  Q_1(x,x_k,\mu)=\tilde f(x_k,\mu)+\nabla\tilde f(x_k,\mu)^T(x-x_k)+\frac{1}{2}(x-x_k)^T \tilde B(x_k,\mu) (x-x_k),
\end{equation}
where $\tilde f(x,\mu)$ is given in \eqref{tildef},
\begin{equation*}\label{B}
\begin{array}{rl}
  \tilde B(x,\mu)=&A^T\text{Diag}\left(\kappa((b-Ax)_1,\mu),\ldots,\kappa((b-Ax)_M,\mu)\right)A\\[5pt]
          =&\displaystyle\sum_{m\in\M}\kappa((b-Ax)_m,\mu)a_ma_m^T,
\end{array}\end{equation*}and $\kappa(\cdot,\mu)$ is given in \eqref{kappa}. By the definition of $\kappa(\cdot,\mu)$, we have \begin{equation}\label{lambdamax}
  \lambda_{\max}\left(\tilde B(x,\mu)\right)\leq \lambda_{\max}\left(\sum_{m\in\M}4q\mu^{q-2}a_ma_m^T\right)\leq 4q\mu^{q-2}\sum_{m\in\M}\left\|a_m\right\|^2.
\end{equation}
Similarly, define the quadratic approximation of $h(\cdot)$ around $x_k$ as
\begin{equation*}\label{quad-problem2}
  Q_2(x,x_k)=h(x_k)+\nabla h(x_k)^T(x-x_k)+\frac{1}{2}L_h^k\|x-x_k\|^2,
\end{equation*}where $L_h^k>0$ is an estimation of $L_h$ in \eqref{L}. Define
\begin{equation}\label{QQ}
\begin{array}{rl}
  Q(x,x_k,\mu)=&Q_1(x,x_k,\mu)+Q_2(x,x_k).
  \end{array}
\end{equation}

The following lemma {indicates} that the convex quadratic function $Q(x,x_k,\mu)$ in \eqref{QQ} is a local upper bound of the smoothing function $\tilde F(x,\mu)$ defined in \eqref{tildeF} around point $x_k$ as long as \begin{equation}\label{hQ}
  h(x)\leq Q_2(x,x_k)
\end{equation}holds true with $L_h^k>0.$

\begin{lemma}\label{lemma-quad}
For any $x_k$ and $x$ such that
\begin{align}
  \left(A(x_k-x)\right)_m &\leq \mu,~m\in\I_{x_k}^{\mu},\label{trust1}\\
  \left(A(x_k-x)\right)_m &\geq \frac{-\left(b-Ax_k\right)_m}{2},~m\in\J_{x_k}^\mu,\label{trust2}
\end{align}where \begin{equation}\label{Mx}\begin{array}{rcl}\I_{x_k}^{\mu}&=&\left\{m\,|\,\left(b-Ax_k\right)_m<-\mu\right\},\\[5pt]
\J_{x_k}^{\mu}&=&\left\{m\,|\,\left(b-Ax_k\right)_m> 2\mu\right\},\end{array}\end{equation}if \eqref{hQ} holds true with $L_h^k>0,$ then
  \begin{equation}\label{keyineq}
    \tilde F(x,\mu) \leq Q(x,x_k,\mu),
  \end{equation}where $Q(x,x_k,\mu)$ is defined in \eqref{QQ}.
\end{lemma}
\begin{proof}
   Recalling the definition of $Q_1(x,x_k,\mu)$ in \eqref{quad-upper} and treating $(b-Ax)_m$ and $(b-Ax_k)_m$ as $t$ and $\hat t$ in (iii) of Lemma \ref{lemma-smooth}, respectively, we get $\tilde f(x,\mu)\leq Q_1(x,x_k,\mu).$ 
    Combining this, \eqref{QQ}, and \eqref{hQ}, we immediately obtain the desired result \eqref{keyineq}.\qed
\end{proof}

Based on Lemma \ref{lemma-quad}, {we propose our SSQP algorithmic framework for solving problem \eqref{pnorm} in Page \pageref{frameworkpage}.
Some remarks on the proposed SSQP algorithmic framework are in order.}

\begin{table} \label{frameworkpage}
\begin{center}
\framebox{
\begin{minipage}{11.8cm}
\centerline{\bf An SSQP Algorithmic Framework for Problem \eqref{pnorm}}
\vspace{0.05cm}
\textbf{Step 1.}
\textbf{Initialization}. Choose the initial feasible point $x_0$ and the parameters $\epsilon\in(0,1],$ $0<\sigma<1,$ $\eta>1,$ and $L_h^{\max}\geq L_h^0\geq L_h^{\min}>0$ with $L_h^{\max}\geq L_h.$ Set $i=0~(\text{outer~iteration~index}),$~$k=0~(\text{inner~iteration~index}),$ and \begin{equation}\label{parameters}\mu_0=\frac{\epsilon}{\sigma^{\lfloor\log_{\sigma}\epsilon\rfloor}}\in(\sigma,\,1].\end{equation}
\textbf{Step 2.} \textbf{Termination of the inner iteration}. Set $\mu=\mu_i.$ If $x_k$ satisfies \begin{equation}\label{termination}\left\|P_{\X}\left(x_k-\nabla \tilde F\left(x_k,\mu\right)\right)-x_k\right\|\leq \mu,\end{equation} go to \textbf{Step 5}; else go to \textbf{Step 3.}\\[2.5pt]
\textbf{Step 3.} \textbf{Calculating the new iterate}. Let $x_{k+1}$ be an (approximate) solution of the following convex QP \begin{equation}\label{quad-problem}
    \begin{array}{rl}
    \displaystyle \min_{x\in\X} & Q(x,x_k,\mu)\\
    \mbox{s.t.} & 
    \displaystyle\left(A(x_k-x)\right)_m \leq \mu,~m\in\I_{x_k}^\mu,\\
                &\left(A(x_k-x)\right)_m \geq -\frac{\left(b-Ax_k\right)_m}{2},~m\in\J_{x_k}^\mu
    \end{array}
  \end{equation} such that
  \begin{equation}\label{framework-decrease}
    \tilde F(x_k,\mu)-\tilde F(x_{k+1},\mu)\geq O(\mu^{4-q}),
  \end{equation}
  where $Q(x,x_k,\mu)$ is defined in \eqref{QQ} and $\I_{x_k}^\mu, \J_{x_k}^\mu$ are defined in \eqref{Mx}, respectively. Compute $s_{k+1}=x_{k+1}-x_{k},~y_{k+1}=\nabla h(x_{k+1})-\nabla h(x_{k}),$ and \begin{equation}\label{rk}r_k=\frac{h\left(x_{k+1}\right)-h(x_k)-\nabla h(x_k)^T\left(x_{k+1}-x_k\right)}{\frac{1}{2}L_h^k\|x_{k+1}-x_k\|^2}.\end{equation}\\[2.5pt]
\textbf{Step 4.} \textbf{Updating the estimated Lipschitz constant}. If $r_k\leq 1,$ compute $L_h^{k+1}$ by \begin{equation}\label{Lhk}L_h^{k+1}=\max\left\{\min\left\{L_h^{\max},\frac{s_{k+1}^Ty_{k+1}}{\|s_{k+1}\|^2}\right\},L_h^{\min}\right\},\end{equation} set $k=k+1,$ and go to \textbf{Step 2}; else set $$L_h^k=\eta L_h^k,$$ and go to \textbf{Step 3}.\\[2.5pt]
\textbf{Step 5.} \textbf{Termination of the outer iteration.} If $\mu\leq \epsilon,$ terminate the algorithm; else go to \textbf{Step 6.}\\[2.5pt]
\textbf{Step 6.} \textbf{Updating the smoothing parameter}. Set $$\mu_{i+1}=\sigma\mu_i,$$$i=i+1,$ $x_0=x_k,$ $k=0,$ and go to \textbf{Step 2}.
\end{minipage}
}
\end{center}
\end{table}

First, 
to solve {the} nonsmooth problem \eqref{pnorm}, the proposed SSQP framework approximately solves a series of smoothing approximation problems \eqref{smoothingproblem} with decreasing smoothing parameters. The solution accuracy of the smoothing approximation problem \eqref{smoothingproblem} is adaptively controlled by \eqref{termination}.

Second, the convex QP problem \eqref{quad-problem} can be efficiently solved (in an exact manner) by the active-set method or the interior-point method\cite{gould04mpqp,nemirovskimodern,yuannumerical,nocedalnumerical,ye98bookinterior}. In fact, performing a simple shrink projection gradient step for solving problem \eqref{quad-problem} in an inexact fashion is sufficient to guarantee \eqref{framework-decrease} (see Lemma \ref{lemma-sufficient}) and hence the worst-case iteration complexity of the proposed framework (See Theorem \ref{thm-iteration}).

Third, the Lipschitz constant $L_h$, when is unknown, is adaptively updated in \textbf{Step 4}, which is also used in \cite{Lu12iterative,wright09tspsparse,birgin00siamoptbb,bb88,dai02imabb,jiangdai13}. If $L_h$ is known, we can set $L_h^{\max}=L_h^{\min}=L_h^0=L_h$ in the proposed algorithm, and {the} $r_k$ in \eqref{rk} satisfies $r_k\leq 1$ at each iteration. Other adaptive ways of updating $L_h$ can also be found in \cite{cartis11mpcubic1,cartis11mpcubic2,cartis11siamoptcomposite,bian13siamoptquadratic}. However, this will not affect the worst-case iteration complexity order of the proposed framework.

Finally, 
the parameter $\mu_0$ in \eqref{parameters} is chosen such that the final smoothing parameter $\mu$ is equal to $\epsilon$ {once the framework is terminated.
This} simplifies the worst-case iteration complexity analysis, but does not affect the worst-case iteration complexity order.

In the following analysis, we {assume, without loss of generality, that} $F(x)\geq 0$ for all $x\in\X.$ This, together with \eqref{tildeFin}, immediately implies that $\tilde F(x,\mu)$ in \eqref{tildeF} with any $\mu\geq0$ satisfies
\begin{equation}\label{lower0}\tilde F(x,\mu)\geq 0,~\forall~x\in\X,~\forall~\mu\geq 0.\end{equation}

Next, we show that the proposed SSQP framework is well defined and will terminate after finitely many iterations.

\begin{lemma}\label{lemmaK0}
  For any $\mu>0$ and $k\geq 0$, \textbf{Step 3} in the proposed SSQP framework will be executed at most \begin{equation}\label{K0}K_0:=\left\lceil\log_{\eta}\frac{L_h}{L_{h}^{\min}}\right\rceil+1\end{equation} times, i.e., the convex QP in the form of \eqref{quad-problem} with any $\mu>0$ and $k\geq 0$ will be (approximately) solved at most $K_0$ times.
\end{lemma}
\begin{proof}
  It follows from \eqref{Lhk} that $L_h^k\geq L_h^{\min}$ for any $k\geq 0.$ Since $$L_h^k\eta^{K_0-1}\geq L_h^{\min}\eta^{K_0-1}\geq L_h,$$ it follows from \eqref{L} that
  {\begin{align*}h(x_{k+1}) &\leq h(x_k)+\nabla h(x_k)^T(x_{k+1}-x_k)+\frac{1}{2}L_h\|x_{k+1}-x_k\|^2\\
  &\leq h(x_k)+\nabla h(x_k)^T(x_{k+1}-x_k)+\frac{1}{2}L_h^k\eta^{K_0-1}\|x_{k+1}-x_k\|^2,
  \end{align*}}which further implies that {the} $r_k$ in \eqref{rk} satisfies $r_k\leq 1.$ {According to \textbf{Step 4} of the SSQP framework, the inner iteration index $k$ will be incremented after solving the convex QP in the form of \eqref{quad-problem} at most $K_0$ times.}  
  \qed
\end{proof}

\begin{lemma}
  For any $k\geq 0$ in the SSQP framework, we have
  \begin{equation}\label{barLh}L_h^k\leq \bar L_h:=\max\left\{L_h^0,\,L_h^{\max},\,\eta L_h\right\}.\end{equation}
\end{lemma}
\begin{proof}
   From \eqref{L}, for any $x\in\X,$ we have
   $$\frac{h\left(x\right)-h(x_k)-\nabla h(x_k)^T\left(x-x_k\right)}{\frac{1}{2}L_h\|x-x_k\|^2}\leq 1.$$ According to \textbf{Step 4} of the SSQP framework, $L_h^k$ is set to be $\eta L_h^k$ only when $r_k>1,$ and in this case there must hold $L_h^k<L_h$. Hence, \eqref{barLh} is true.
\end{proof}

The following Lemma \ref{lemma-sufficient} guarantees {the existence of $x_{k+1}$ satisfying the relation \eqref{framework-decrease}.}

\begin{lemma}\label{lemma-sufficient}
  For any $\mu\in(0,1]$ and $k\geq 0$ in the proposed SSQP framework, suppose that 
  \begin{itemize}
    \item [-] $x_{k+1}^{\text{exact}}$ is the solution of problem \eqref{quad-problem},
    \item [-] $x_{k+1}^{\text{snorm}}$ is the solution of the following problem \begin{equation}\label{quad-problem-framework}
    \begin{array}{rl}
    \displaystyle \min_{x\in\X} & Q(x,x_k,\mu)\\[5pt]
    \mbox{s.t.} & 
    \displaystyle \left\|A\left(x-x_k\right)\right\|_{\infty}\leq \mu,
    \end{array}
  \end{equation}
   \item [-] and \begin{equation}\label{solutionprojection} x_{k+1}^{\text{proj}}=x_k+\xi_k\tau_k d_k,\end{equation} where
   \begin{equation}\label{tauk}\tau_k=\frac{\mu}{\left(\max_m\left\{\|a_m\|\right\}+1\right)\|d_k\|}, \end{equation}
     \begin{equation}\label{univariateqp}\xi_k =\min\left\{\frac{-d_k^T\nabla\tilde F\left(x_k,\mu\right) }{
\tau_k d_k^T\left(\tilde B_k+L_hI_N\right) d_k},1\right\},\end{equation}
and $$d_k=P_{\X}(x_k-\nabla\tilde F\left(x_k,\mu\right))-x_k.$$

  \end{itemize}If \eqref{termination} is not satisfied, then
  \begin{align}
\tilde F(x_k,\mu)-\tilde F(x_{k+1}^{\text{exact}},\mu)\geq\tilde F(x_k,\mu)-\tilde F(x_{k+1}^\text{snorm},\mu)\geq \tilde F(x_k,\mu)-\tilde F(x_{k+1}^\text{proj},\mu)\geq \frac{\mu^{4-q}}{J_0},\label{fdecrease} 
\end{align}where \begin{equation}\label{barJ}J_0=\max\left\{8q\sum_{m}\|a_m\|^2+2\bar L_h,2\max_m\left\{\|a_m\|\right\}+2\right\},\end{equation} and $\bar L_h$ is given in \eqref{barLh}.
\end{lemma}

\begin{proof}For simplicity, {denote $\tilde B(x_k,\mu)$ and $\nabla\tilde F\left({x_k,\mu}\right)$ by $\tilde B_k$ and $\nabla\tilde F_k$, respectively in the proof.} Since $d_k=P_{\X}(x_k-\nabla\tilde F_k)-x_k,$ it follows from the property of projection that
\begin{equation}\label{projectionproperty}
-\nabla\tilde F_k^Td_k\geq \|d_k\|^2\geq 0.
\end{equation} This implies that the $\xi_k$ in \eqref{univariateqp} satisfies $\xi_k\in[0,1].$

Next, we first show the last inequality $\tilde F(x_k,\mu)-\tilde F(x_{k+1}^\text{proj},\mu)\geq {\mu^{4-q}}/{J_0}$ in \eqref{fdecrease} holds true.


We claim that $x_k\left(\xi\right):=x_k+\xi\tau_k d_k$ 
is feasible to problem \eqref{quad-problem} for all $\xi\in[0,1].$  
First of all, since \eqref{termination} is not satisfied, we have \begin{equation}\label{termination2}\left\|d_k\right\|> \mu.\end{equation}
By this and \eqref{tauk}, we get {\begin{align*}
  \tau_k\leq&\frac{\mu}{\left(\max_m\left\{\|a_m\|\right\}+1\right)\mu} =  \frac{1}{\max_m\left\{\|a_m\|\right\}+1}  \leq  1.
\end{align*}}Hence, $x_k\left(\xi\right)=\left(1-\xi\tau_k\right)x_k+\xi\tau_k P_{\X}(x_k-\nabla\tilde F_k)$ is a convex combination of $x_k\in \X$ and $P_{\X}(x_k-\nabla\tilde F_k)\in \X$.
By the convexity of $\X,$ we have
\begin{equation}\label{feasible1}x_k(\xi)\in \X,~\forall~\xi\in[0,1].\end{equation} Moreover, by the definition \eqref{tauk} of $\tau_k,$ 
we have
\begin{equation}\label{feasible2}\left|\left(A(x_k- x_k(\xi))\right)_m\right|=\left|a_m^T(x_k-x_k(\xi))\right|=\xi\tau_k|a_m^T d_k|\leq \xi\tau_k\left\|a_m\right\| \left\|d_k\right\|\leq \mu\end{equation} for all $\xi\in[0,1]$ and $m\in\M.$ This shows that
 $x_k\left(\xi\right)$ satisfies \eqref{trust1} and \eqref{trust2} for all $\xi\in[0,1]$. Hence, $x_k\left(\xi\right)$ is feasible to problem \eqref{quad-problem} for all $\xi\in[0,1]$ and hence $x_{k+1}^\text{proj}$ in \eqref{solutionprojection} (due to $\xi_k\in[0,1]$).

Now, we consider the univariate box constrained QP problem \begin{equation}\label{oneqp}\xi_k =\arg\min_{0\leq \xi\leq 1} Q(x_k+\xi \tau_k d_k,x_k,\mu),\end{equation} which admits a closed-form solution \eqref{univariateqp}. 
We first consider the case $\xi_k=1,$ which implies that ${-\nabla \tilde F_k^T d_k}{}\geq
\tau_k d_k^T\left(\tilde B_k+L_hI_N\right) d_k.$ Here, we have
\begin{equation}\label{decrease1}
\begin{array}{rl}
  \tau_k\nabla \tilde F_k^T d_k+\frac{1}{2}\tau_k^2d_k(\tilde B_k+L_hI_N)d_k \leq   \displaystyle\frac{\tau_k}{2}\nabla \tilde F_k^T d_k                                                                 \leq\displaystyle-\frac{\tau_k}{2}\|d_k\|^2 =\displaystyle\frac{-\mu\|d_k\|}{2\left(\max\left\{\|a_m\|\right\}+1\right)},
\end{array}\end{equation}where the second inequality is due to \eqref{projectionproperty}.
For the other case where $\xi_k=\frac{-\nabla \tilde F_k^T d_k}{
\tau_k d_k^T\left(\tilde B_k+L_hI_N\right) d_k}$, we have
\begin{equation}\label{decrease2}
\begin{array}{rl}\displaystyle\xi_k\tau_k\nabla \tilde F_k^T d_k+\frac{1}{2}\xi_k^2\tau_k^2d_k(\tilde B_k+L_hI_N)d_k=&\displaystyle-\frac{\left(\nabla \tilde F_k^T d_k\right)^2}{2 d_k^T\left(\tilde B_k+L_hI_N\right) d_k}\\[15pt]
                                                                 \leq & \displaystyle-\frac{\|d_k\|^4}{2\lambda_{\max}\left(\tilde B_k+L_hI_N\right)\|d_k\|^2} \\[15pt]
                                                                 \leq & \displaystyle-\frac{\|d_k\|^2\mu^{2-q}}{2\left(4q\sum_{m}\|a_m\|^2+L_h\right)},
\end{array}\end{equation}where the first inequality is due to \eqref{projectionproperty} and the last inequality is due to \eqref{lambdamax}. Combining \eqref{keyineq}, \eqref{decrease1}, and \eqref{decrease2}, we obtain
\begin{align*}
\tilde F(x_k,\mu)-\tilde F(x_{k+1}^{\text{proj}},\mu)=& \tilde F(x_k,\mu)-\tilde F(x_{k}+\xi_k\tau_kd_k,\mu)\\
\geq& \tilde F(x_k,\mu)-Q(x_{k}+\xi_k\tau_kd_k,\mu)\\
=&-\xi_k\tau_k\nabla \tilde F_k^T d_k-\frac{1}{2}\xi_k^2\tau_k^2d_k(\tilde B_k+L_h)d_k\\
\geq & \min\left\{\frac{\mu\|d_k\|}{2\left(\max\left\{\|a_m\|\right\}+1\right)},\frac{\|d_k\|^2\mu^{2-q}}{2\left(4q\sum_{m}\|a_m\|^2+L_h\right)}\right\}.
\end{align*}This, together with \eqref{barLh}, \eqref{barJ}, and \eqref{termination2}, immediately implies the desired result $$\tilde F(x_k,\mu)-\tilde F(x_{k+1}^\text{proj},\mu)\geq {\mu^{4-q}}/{J_0}.$$

 Now we show the first two inequalities in \eqref{fdecrease}. From the above analysis (cf. \eqref{feasible1} and \eqref{feasible2}), we know that $x_{k+1}^\text{proj}$ is feasible to problem \eqref{quad-problem-framework}. Since $x_{k+1}^{snorm}$ is the solution of problem \eqref{quad-problem-framework}, it follows that $$\tilde F(x_k,\mu)-\tilde F(x_{k+1}^\text{snorm},\mu)\geq \tilde F(x_k,\mu)-\tilde F(x_{k+1}^\text{proj},\mu).$$ Moreover, since the feasible region of problem \eqref{quad-problem-framework} is a subset of the one of problem \eqref{quad-problem}, we immediately get $$\tilde F(x_k,\mu)-\tilde F(x_{k+1}^{\text{exact}},\mu)\geq\tilde F(x_k,\mu)-\tilde F(x_{k+1}^\text{snorm},\mu).$$
The proof is completed.\qed

\end{proof}

As shown in Lemma \ref{lemma-sufficient}, for the next iterate $x_{k+1}$ to achieve a decrease of order $O(\mu^{4-q})$ as required in \eqref{framework-decrease}, problem \eqref{quad-problem} is not necessarily to be solved in an exact manner; a simple shrink projection gradient step, i.e., $x_{k+1}^{\text{proj}}=x_k+\xi_k\tau_kd_k,$ suffices to satisfy \eqref{fdecrease}. This gives the flexibility to choose subroutines for solving problem \eqref{quad-problem} inexactly.

It is also worthwhile remarking that Lemma \ref{lemma-sufficient} holds true for any convex set $\X$, which is not necessarily a polyhedron. If $\X=\mathbb{R}^N,$ then problem \eqref{quad-problem-framework} is a trust region subproblem with a scaled infinity norm constraint. The infinity norm in \eqref{quad-problem-framework} could be replaced by the Euclidean norm, and the solution to the corresponding counterpart still satisfies \eqref{framework-decrease}. 

Without loss of generality, we focus on analyzing the SSQP framework when {the} $x_{k+1}$ in \textbf{Step 3} is chosen to be $x_{k+1}^{\text{proj}}$ in \eqref{solutionprojection} in the rest part of this section.

The following lemma {states that} the inner loop termination criterion \eqref{termination} of the SSQP framework {can be} satisfied after finite number of iterations.


\begin{lemma}\label{lemma-iteration-unknown}
  Let $x_{k+1}=x_{k+1}^{\text{proj}}$ in the proposed SSQP framework. 
Then, for any $\mu\in(0,1]$, \textbf{Step 2} of the proposed SSQP framework will be executed at most $$\left\lceil \tilde F(x_0,1) J_0 \mu^{q-4}\right\rceil$$ times,
   and the inner termination criterion \eqref{termination} is satisfied after at most $$\left\lceil \tilde F(x_0,1)J_0 K_0 \mu^{q-4}\right\rceil$$ iterations, where $J_0$ and $K_0$ are given in \eqref{barJ} and \eqref{K0}, respectively.

\end{lemma}

\begin{proof}
  For any $\mu\in(0,1]$, we use $\hat k$ to denote the first inner iteration index such that $x_{\hat k}$ satisfies \eqref{termination}. Then \eqref{termination2} holds true for all $k\leq\hat k-1.$ Combining this and Lemma \ref{lemma-sufficient}, we obtain
   $$\tilde F(x_0,\mu)-\tilde F(x_{\hat k},\mu)=\sum_{k=1}^{\hat k} \left(\tilde F(x_{k-1},\mu)-\tilde F(x_{k},\mu)\right) \geq \hat k \frac{\mu^{4-q}}{J_0}.$$
   By using \eqref{lower0}, we conclude 
    $$\hat k \leq  \tilde F(x_0,\mu)J_0 \mu^{q-4}\leq \tilde F(x_0,1)J_0 \mu^{q-4}.$$
   This, together with Lemma \ref{lemmaK0}, immediately implies the second statement of Lemma \ref{lemma-iteration-unknown}. \qed
\end{proof}

Now, we are ready to show that the proposed SSQP framework terminates after finite number of iterations.

\begin{theorem}\label{thm-iteration}
Let $x_{k+1}=x_{k+1}^{\text{proj}}$ in the proposed SSQP framework. Then, for any $\epsilon\in(0,1],$ the framework will terminate within at most
  \begin{equation}\label{com-func}\left\lceil J_{T}^q \epsilon^{q-4}\right\rceil\end{equation} iterations,~
where \begin{equation}\label{JT}\displaystyle J_{T}^q=\frac{\sigma^{q-4}\left(\tilde F(x_0,1)J_0K_0+1\right)}{{\sigma^{q-4}-1}}, \end{equation} and $K_0$ and $J_0$ are defined in \eqref{K0} and \eqref{barJ}, respectively.
\end{theorem}
\begin{proof}
  Define $I_0=\left\lfloor \log_{\sigma}\epsilon\right\rfloor.$ According to the SSQP framework, we have \begin{equation}\label{mu}\mu_i=\mu_0\sigma^i\geq \mu_0\sigma^{I_0}=\epsilon,~\forall~i=0,1,\dots,I_0.\end{equation} In particular, we have
\begin{equation}\label{muI0}
\mu_{I_0}=\mu_0\sigma^{I_0}=\epsilon.
\end{equation} By \eqref{mu} and Lemma \ref{lemma-iteration-unknown}, for any fixed $\mu>0,$ the number of iterations that the SSQP framework takes to return a point satisfying \eqref{termination} is at most
$$\left\lceil \tilde F(x_0,1)J_0K_0 \mu^{q-4}\right\rceil=\left\lceil \tilde F(x_0,1)J_0K_0 \left(\mu_0\sigma^i\right)^{q-4}\right\rceil.$$
 Therefore,
the total number of iterations for the proposed framework to terminate is at most
$$\sum_{i=0}^{I_0} \left\lceil \tilde F(x_0,1)J_0K_0 \left(\mu_0\sigma^i\right)^{q-4}\right\rceil\leq \left(\tilde F(x_0,1)J_0 K_0+1\right)\mu_0^{q-4}\frac{\sigma^{\left(q-4\right)\left(I_0+1\right)}-1}{\sigma^{q-4}-1} \leq J_{T}^q \epsilon^{q-4},$$ where the last inequality is due to \eqref{JT} and \eqref{muI0}. \qed
\end{proof}

The worst-case iteration complexity function in \eqref{com-func} is a strictly decreasing function with respect to $q\in(0,1)$ for fixed $\epsilon\in(0,1).$
This is intuitive because problem \eqref{pnorm} becomes more difficult to solve as $q$ decreases.


\subsection{Worst-Case Iteration Complexity Analysis}\label{subsec-analysis}
In this subsection, we show that the point returned by the proposed SSQP framework is an $\epsilon$-KKT point of problem \eqref{pnorm}. 
To do this, we need to give the definition of the $\epsilon$-KKT point first. Our definition of the $\epsilon$-KKT point of problem \eqref{pnorm} is given as follows, which is a perturbation of the KKT optimality conditions in Theorem \ref{thm-optcondition}. 

\begin{definition}[$\epsilon$-KKT point]\label{definitione}
  For any given $\epsilon>0,\,$ $\bar x\in\X$ is called an $\epsilon$-KKT point of problem \eqref{pnorm} if there exists 
  $\bar\lambda\geq0\in\mathbb{R}^{|\K_{\bar x}^{\epsilon}|}$ such that
  \begin{equation}\label{complementary}
    \left|\bar\lambda_m(b-A\bar x)_m\right|\leq \epsilon^q,~m\in \K_{\bar x}^{\epsilon}
  \end{equation} and
  \begin{equation}\label{ekkt1}
      \left\|\bar x-P_{\X}\left(\bar x-\nabla L^{\epsilon}(\bar x,\bar\lambda)\right)\right\| \leq \epsilon,
    \end{equation}
    where \begin{equation}\label{Lagepsilon}\begin{array}{rl}L^{\epsilon}(x,\lambda)=&\displaystyle \sum_{m\in\J_{\bar x}^{\epsilon}}(b-Ax)_m^q+h(x)\displaystyle +\sum_{m\in\K_{\bar x}^{\epsilon}}\lambda_m(b-Ax)_m\end{array}\end{equation} 
    with
    \begin{equation}\label{barM}
  \begin{array}{rcl}
 \I_{\bar x}^{\epsilon}&=&\left\{m\,|\,(b-A\bar x)_m<-\epsilon\right\},\\
  \J_{\bar x}^{\epsilon}&=&\left\{m\,|\,(b-A\bar x)_m>\epsilon\right\},\\
  \K_{\bar x}^{\epsilon}&=&\left\{m\,|\,-\epsilon\leq (b-A\bar x)_m\leq \epsilon\right\}.\end{array}\end{equation}
\end{definition}

Notice that if $\epsilon=0$ in \eqref{complementary}, \eqref{ekkt1}, and \eqref{barM}, 
then the $\epsilon$-KKT point in Definition \ref{definitione} reduces to the KKT point of problem \eqref{pnorm} (cf. Theorem \ref{thm-optcondition}).

The following definition of the $\epsilon$-KKT point for problem {\eqref{chenniu}} has been used in \cite{chen13siamopttrust,bianchen14mpfeasible}.
\begin{definition}\label{def:ekktother}
  For any $\epsilon\in(0,1],$ $\bar x$ is called an $\epsilon$-KKT point of problem \eqref{chenniu} if it satisfies \begin{equation}\label{def-ekkt-other}\left\|\left(Z_{\bar x}^{\epsilon}\right)^T\nabla F_{{\bar x}}^{\epsilon}(\bar x)\right\|_{\infty}\leq \epsilon,\end{equation}where
  \begin{equation*}\label{reducedFepsilon}F_{{\bar x}}^{\epsilon}(x)=\sum_{\left|a_m^T{\bar x}\right|> \epsilon} \left|a_m^Tx\right|^q+h(x)\end{equation*} 
  and $Z_{\bar x}^{\epsilon}$ is the matrix whose columns form an orthogonal basis for the null space of $\left\{a_m\,|\,\left|a_m^T\bar x\right|\leq \epsilon\right\}.$
\end{definition}

Our definition of the $\epsilon$-KKT point in Definition \ref{definitione} reduces to Definition \ref{definitionereduce} when problem \eqref{pnorm} reduces to problem \eqref{chenniu}.

\begin{definition}\label{definitionereduce}
  For any given $\epsilon>0,\,$ $\bar x$ is called an $\epsilon$-KKT point of problem \eqref{chenniu} if there exists 
  $\bar\lambda\in\mathbb{R}^{|\hat\K_{\bar x}^{\epsilon}|}$ such that
  \begin{equation}\label{complementaryreduce}
    \left|\bar\lambda_ma_m^T\bar x\right|\leq \epsilon^q,~m\in \hat\K_{\bar x}^{\epsilon}
  \end{equation} and
  \begin{equation}\label{ekkt1reduce}
      \left\|\nabla \hat L^{\epsilon}(\bar x,\bar\lambda)\right\| \leq \epsilon,
    \end{equation}where \begin{equation*}\begin{array}{rl}\hat L^{\epsilon}(x,\lambda)=&\displaystyle \sum_{m\in\hat\I_{\bar x}^{\epsilon}}(-a_m^Tx)^q+\sum_{m\in\hat\J_{\bar x}^{\epsilon}}(a_m^Tx)^q+h(x)\displaystyle +\sum_{m\in\hat\K_{\bar x}^{\epsilon}}\lambda_m(b-Ax)_m\end{array}\end{equation*} 
    with
    \begin{equation}\label{hatMspe}
  \begin{array}{rcl}
 \hat\I_{\bar x}^{\epsilon}&=&\left\{m\,|\,a_m^T\bar x<-\epsilon\right\},\\
  \hat\J_{\bar x}^{\epsilon}&=&\left\{m\,|\,a_m^T\bar x>\epsilon\right\},\\
  \hat\K_{\bar x}^{\epsilon}&=&\left\{m\,|\,-\epsilon\leq a_m^T\bar x\leq \epsilon\right\}.\end{array}\end{equation}
\end{definition}

\begin{remark}\label{remark-ekkt}
   The $\epsilon$-KKT point for problem \eqref{chenniu} in Definition \ref{definitionereduce} is stronger than the one in Definition \ref{def:ekktother}. On one hand, it is clear that \eqref{ekkt1reduce} implies \eqref{def-ekkt-other}. On the other hand, if \eqref{def-ekkt-other} is true, then there must exist $\bar \lambda$ such that \eqref{ekkt1reduce} is satisfied\footnote{Here, the differences between two norms ($\|\cdot\|_{\infty}$ in \eqref{def-ekkt-other} and $\|\cdot\|$ in \eqref{ekkt1reduce}) are neglected.}. However, this $\bar \lambda$ does not necessarily satisfy \eqref{complementaryreduce}.
\end{remark}

In the next, we show that the point returned by the proposed SSQP framework is an $\epsilon$-KKT point of problem \eqref{pnorm} defined in Definition \ref{definitione}.

For any given $\epsilon\in(0,1],$ let $\bar x$ be the point returned by the proposed SSQP framework. When the framework is terminated, there holds $\mu=\epsilon$ (cf. \eqref{muI0}) Then, it follows from \eqref{termination} that $\bar x\in \X$ satisfies \begin{equation}\label{terminatione}\left\|\nabla \tilde F\left(\bar x,\epsilon\right)\right\|\leq \epsilon.\end{equation} Define $\I_{\bar x}^\epsilon,\J_{\bar x}^\epsilon,$ and $\K_{\bar x}^\epsilon$ as in \eqref{barM},~
and
\begin{equation}\label{barlambda}\displaystyle \bar\lambda_m= \left[\theta^q(t,\epsilon)\right]'_{t=(b-A\bar x)_m},~m\in \K_{\bar x}^\epsilon.\end{equation}It is obvious that $\bar\lambda_m\geq0$ for all $m\in\K_{\bar x}^\epsilon.$%


\begin{theorem}\label{thm-e}
  For any $\epsilon\in(0,1],$ let $\bar x$ be the point returned by the proposed SSQP framework and $\bar\lambda$ be defined in \eqref{barlambda}. 
  Then $\bar x$ and $\bar\lambda$ satisfy \eqref{complementary} and \eqref{ekkt1}. 
\end{theorem}

%
%
%

\begin{proof}



Let us first show that $\bar x$ and $\bar\lambda$ satisfy \eqref{complementary}. 
\begin{itemize}
  \item [-] For any $m\in\K_{\bar x}^{\epsilon}$ with $-\epsilon\leq (b-A\bar x)_m\leq 0,$ it follows from \eqref{q1d} and \eqref{barlambda} that $\bar \lambda_m=0,$ and thus $\left|\bar\lambda_m(b-A\bar x)_m\right|=0\leq \epsilon^q;$
  \item [-] For any $m\in\K_{\bar x}^{\epsilon}$ with $0< (b-A\bar x)_m\leq \epsilon,$ we have
  \begin{align*}
\left|\bar\lambda_m(b-A\bar x)_m\right|=&{q\theta^{q-1}((b-A\bar x)_m,\epsilon)}\frac{(b-A\bar x)_m}{\epsilon}(b-A\bar x)_m\leq  q\left(\frac{\epsilon}{2}\right)^{q-1}\epsilon\leq \epsilon^{q},
\end{align*}
where the equality comes from \eqref{q1d} and \eqref{barlambda}, the first inequality is due to \eqref{lowerq} and $0< (b-A\bar x)_m\leq \epsilon,$ and the second inequality is due to the fact that $q2^{1-q}\leq 1$ for all $q\in (0,1).$
\end{itemize}


Now we show that $\bar x$ and $\bar\lambda$ satisfy \eqref{ekkt1}.
~By \eqref{terminatione} and {the} nonexpansive property of the projection operator $P_{\X}(\cdot),$ we get
\begin{align}
  & \left\|\bar x-P_{\X}(\bar x-\nabla L^{\epsilon}(\bar x,\bar\lambda))\right\|\nonumber\\
  \leq & \left\|\bar x-P_{\X}(\bar x-\nabla\tilde F(\bar x,\epsilon))\right\|+\left\|P_{\X}(\bar x-\nabla\tilde F(\bar x,\epsilon))-P_{\X}(\bar x-\nabla L^{\epsilon}(\bar x,\bar\lambda))\right\|\nonumber\\
  \leq & \epsilon + \|\nabla\tilde F(\bar x,\epsilon)-\nabla L^{\epsilon}(\bar x,\bar\lambda)\|\label{ineqcontraction}.
\end{align}
From \eqref{tildeg}, \eqref{Lagepsilon}, and \eqref{barlambda}, we have 
$\nabla\tilde F(\bar x,\epsilon)-\nabla L^{\epsilon}(\bar x,\bar\lambda)=0.$ 
Combining this with \eqref{ineqcontraction} immediately yields \eqref{ekkt1}. The proof is completed.
\qed \end{proof}

By combining Theorems \ref{thm-iteration} and \ref{thm-e}, we obtain the following worst-case iteration complexity result. 

\begin{theorem}\label{thm-complexity2}
For any $\epsilon\in(0,1]$, 
the total number of iterations for the SSQP framework to return an $\epsilon$-KKT point of problem \eqref{pnorm} satisfying \eqref{complementary} and \eqref{ekkt1} is at most $$O\left(\epsilon^{q-4}\right).$$ In particular, {letting} $x_{k+1}$ be $x_{k+1}^{\text{proj}},$ $x_{k+1}^{\text{snorm}},$ or $x_{k+1}^{\text{exact}}$ in the proposed SSQP framework, the total number of iterations for the framework to return an $\epsilon$-KKT point of problem \eqref{pnorm} satisfying \eqref{complementary} and \eqref{ekkt1} is at most $$\left\lceil J_T^q \epsilon^{q-4}\right\rceil,$$ where $J_T^q$ is given in \eqref{JT}.
%
\end{theorem}

As a direct consequence of Theorem \ref{thm-complexity2}, we have the following corollary.

\begin{corollary}\label{corollary-chenniu} For any $\epsilon\in(0,1]$, let $\bar x$ be the point returned by the SSQP framework when applied to solve problem \eqref{chenniu}, and define
$$\bar\lambda_m=\text{sgn}\left(a_m^T\bar x\right)\left[\theta^q(t,\epsilon)\right]'_{t=\left|a_m^T\bar x\right|},~m\in\hat\K_{\bar x}^{\epsilon},$$ where $\hat\K_{\bar x}^{\epsilon}$ is given in \eqref{hatMspe}. Then, $\bar x$ and $\bar \lambda$ satisfy \eqref{complementaryreduce} and \eqref{ekkt1reduce}. Moreover, the total number of iterations for the SSQP framework to return the $\epsilon$-KKT point $\bar x$ of problem \eqref{chenniu} is at most $$O\left(\epsilon^{q-4}\right).$$
\end{corollary}



\subsection{Comparisons of SSQP with SQR}

In this subsection, we compare the proposed SSQP framework with the SQR algorithms proposed in \cite{bianchen14mpfeasible} and \cite{chen13siamopttrust} for solving problem \eqref{chenniu} (possibly with box constraints) and problem \eqref{bianchen}, respectively.

First of all, the SSQP algorithmic framework is designed for solving a more difficult problem, i.e., problem \eqref{pnorm} with {a composite non-Lipschitzian objective and a general polyhedral constraint,} which includes problems \eqref{chenniu} and \eqref{bianchen} as special cases.

Now, we give a detailed comparison of the SQR algorithm in \cite{bianchen14mpfeasible} and the proposed SSQP framework with $x_{k+1}$ chosen to be $x_{k+1}^{\text{proj}}$ at each iteration from the perspective of iteration complexity and solution quality when both of them are applied to solve {the} unconstrained problem \eqref{chenniu}; see Table \ref{table}, where $\tilde F(x,\epsilon)$ reduces to
$$\tilde F(x,\epsilon)=\sum_{m\in\M}\left(\theta^q\left(a_m^Tx,\epsilon\right)+\theta^q\left(-a_m^Tx,\epsilon\right)\right)+h(x).$$ 
%

\begin{table}[h]
\begin{center}
\caption{ {\bf Comparisons of the SQR algorithm in \cite{bianchen14mpfeasible} and
the proposed SSQP framework with $x_{k+1}=x_{k+1}^{\text{proj}}$ in \eqref{solutionprojection} for solving problem \eqref{chenniu}.}}\label{table}
\begin{tabular}{|c|c|c|c|}\hline\hline
\multicolumn{2}{|c|}{} & SQR \cite{bianchen14mpfeasible} & SSQP \\[5pt]
\hline \multirow{3}*{complexity}
        &  iteration number   & $O(\epsilon^{-2})$   &  $O(\epsilon^{q-4})$  \\[5pt]
      \cline{2-4}
       & subproblem per iteration  &  $n$-dimensional QP (12) \cite{bianchen14mpfeasible}   & univariate QP \eqref{oneqp}     \\[5pt]
      \hline \multirow{3}*{quality}
       & optimality residual I & $\left\|\left(Z_{\bar x}^{\epsilon}\right)^T\nabla F_{{\bar x}}^{\epsilon}(\bar x)\right\|_{\infty}\leq \epsilon$    &  $\left\|\nabla \hat L^{\epsilon}(\bar x,\bar\lambda)\right\| \leq \epsilon$    \\[5pt]

        \cline{2-4}
        & optimality residual II   &  $\left\|\nabla \tilde F\left(\bar x,\epsilon\right)\right\|=O\left(\epsilon^{2-2/q}\right)$   &  $\left\|\nabla \tilde F\left(\bar x,\epsilon\right)\right\|\leq \epsilon$  \\[5pt]

        \cline{2-4}

        & complementary violation    &  not guaranteed   &  $\left|\bar\lambda_ma_m^T\bar x\right|\leq \epsilon^q,~m\in \hat\K_{\bar x}^{\epsilon}$    \\[5pt]
\hline\hline
\end{tabular}\end{center}
\end{table}


It is shown in \cite{bianchen14mpfeasible} that the SQR algorithm returns an $\epsilon$-KKT point $\bar x$ satisfying \eqref{def-ekkt-other} within $O(\epsilon^{-2})$ iterations. The SSQP framework with $x_{k+1}=x_{k+1}^{\text{proj}}$, when applied to solve problem \eqref{chenniu}, can return an $\epsilon$-KKT point $\bar x$ satisfying \eqref{complementaryreduce} and \eqref{ekkt1reduce} in no more than $O(\epsilon^{q-4})$ iterations (see Corollary \ref{corollary-chenniu}). Here, one iteration in the SQR algorithm needs solving exactly an $n$-dimensional box constrained QP (problem (12) in \cite{bianchen14mpfeasible}), and the exact solution of the QP subproblem is necessary for Lemma 3 there to hold true. Since the condition number of the quadratic objective in (12) increases asymptotically with $O(\mu^{q-2})$ as the smoothing parameter $\mu$ decreases, the $n$-dimensional box constrained QP (12) becomes more and more difficult to solve. In contrast, one iteration in the SSQP framework with $x_{k+1}=x_{k+1}^{\text{proj}}$ only needs solving approximately {the} QP problem \eqref{quad-problem}. As shown in Lemma \ref{lemma-sufficient}, a good approximate solution \eqref{solutionprojection} of problem \eqref{quad-problem} can be obtained by solving an univariate box constrained QP \eqref{oneqp}, which admits a closed-form solution \eqref{univariateqp}.  



From the perspective of solution quality, the SSQP framework (with $x_{k+1}=x_{k+1}^{\text{proj}}$) actually returns a better ``solution'' compared to the SQR algorithm. As discussed in Remark \ref{remark-ekkt}, the $\epsilon$-KKT point returned by the SSQP framework is stronger than the one returned by the SQR algorithm. In addition, in terms of the residual of smoothing problem, the SSQP framework actually returns an $\epsilon$-KKT point $\bar x$ satisfying $\left\|\nabla \tilde F\left(\bar x,\epsilon\right)\right\|\leq \epsilon$, while the SQR algorithm outputs an $\epsilon$-KKT point $\bar x$ with
  $\left\|\nabla \tilde F\left(\bar x,\epsilon\right)\right\|=O\left(\epsilon^{2-2/q}\right).$

Finally, we remark that the proposed SSQP framework can be directly applied to solve problem \eqref{pnorm} with $q=1$ and is guaranteed to return an $\epsilon$-Clarke KKT point of problem \eqref{pnorm} within $O(\epsilon^{-3})$ iterations. The worst-case iteration complexity of the proposed SSQP framework for computing an $\epsilon$-Clarke KKT point of problem \eqref{pnorm} with $q=1$ is thus the same as the one of the SQR$_1$ algorithm for problem \eqref{bianchen} with $q=1$ in \cite{bian13siamoptquadratic} and better than $O(\epsilon^{-3}\log\epsilon^{-1})$ of the smoothing direct search algorithm for unconstrained Lipschitzian minimization problems in \cite{garmanjani12ima}.

In the {following}, we first extend the definition of $\epsilon$-Clarke KKT point for unconstrained locally Lipschitz continuous optimization problem in \cite{garmanjani12ima} to constrained locally Lipschitz continuous optimization problem \eqref{pnorm} with $q=1,$ and then present the worst-case iteration complexity of the proposed SSQP framework for obtaining such an $\epsilon$-Clarke KKT point.

%
%
%
\begin{definition}\label{def-kkt-cons}
  The point $\bar x$ is called an $\epsilon$-Clarke KKT point of problem \eqref{pnorm} with $q=1$ if it satisfies 
  \begin{equation*}\label{kktq=1}\left\|P_{\X}\left(\bar x-\nabla \tilde F(\bar x,\mu)\right)-\bar x\right\|\leq \epsilon~\text{and}~\mu\leq \epsilon.\end{equation*}
\end{definition}

Using the same argument as in the proof of Theorem \ref{thm-iteration}, we can show the following iteration complexity result.

\begin{theorem}
For any $\epsilon\in(0,1],$ 
the total number of iterations for the SSQP framework to return an $\epsilon$-Clarke KKT point of problem \eqref{pnorm} with $q=1$
is at most $O\left(\epsilon^{-3}\right).$
\end{theorem}

\section{Concluding Remarks} \label{sec-conclusion}
{In this paper, we considered {the} composite nonsmooth nonconvex {non-Lipschitzian} $L_q~(0<q<1)$ minimization problem \eqref{pnorm} over a general polyhedral set. We derived KKT optimality conditions for problem \eqref{pnorm}. These conditions unify various optimality conditions for {non-Lipschitzian} optimization problems developed in \cite{bian13siamoptquadratic,Chen10siamoptlowerbound,ge11mpnote,chen13siamopttrust,bianchen14mpfeasible}. Moreover, we extended the lower bound theory originally developed for local minimizers of unconstrained problem \eqref{bianchen} in \cite{Chen10siamoptlowerbound,Lu12iterative} to constrained problem \eqref{pnorm}. In addition, we proposed an SSQP framework for solving problem \eqref{pnorm} and showed that the proposed framework is guaranteed to return an $\epsilon$-KKT point of problem \eqref{pnorm} satisfying \eqref{complementary} and \eqref{ekkt1} within $O(\epsilon^{q-4})$ iterations. To the best of our knowledge, this is the first algorithmic framework for the {polyhedral constrained} composite $L_q$ minimization with worst-case iteration complexity analysis. The proposed SSQP framework can {directly be} applied to solve problem \eqref{pnorm} with $q=1$ and its worst-case iteration complexity for returning an $\epsilon$-Clarke KKT point is $O(\epsilon^{-3}).$}

Although we focused on {the} $L_q$ minimization problem \eqref{pnorm} in this paper, the techniques developed here can be useful for developing and analyzing algorithms for problems with other regularizers such as the ones given in Appendix A of \cite{bian13siamoptquadratic}. Moreover, most of the results presented in this paper can be easily generalized to problem \eqref{pnorm} where the unknown variable is a positive semidefinite matrix.

\section*{Appendix A: Three Motivating Applications}

\textbf{Support Vector Machine\cite{vapnik92svm,vapnik95svm}.} The support vector machine (SVM) is a state-of-the-art classification method introduced by Boser, Guyon, and Vapnik in 1992 in \cite{vapnik92svm}. Given a database $\left\{s_m\in \mathbb{R}^{N-1},\,y_m\in\mathbb{R}\right\}_{m=1}^M,$ where $s_m$ is called pattern or example and $y_m$ is the label associated with $s_m.$ For convenience, we assume the labels are $+1$ {for positive examples} and $-1$ for negative examples. If the data are linearly separable, the task of SVM is to find a linear discriminant function of the form $\ell(s)=\hat s^Tx$ with $\hat s=[s^T,1]^T\in \mathbb{R}^{N}$ such that all data are correctly classified and at the same time the margin of the hyperplane $\ell$ {that separates the two classes of examples} is maximized. Mathematically, the above problem can be formulated as
\begin{equation}\label{svmseparable}
\begin{array}{ll}
\displaystyle \min_{x} & \displaystyle \frac{1}{2}\sum_{n=1}^{N-1}x_n^2 \\
\mbox{s.t.} & \displaystyle y_m\hat s_m^Tx\geq 1,~m=1,2,\ldots,M.
\end{array}
\end{equation}In practice, data are often not linearly separable. In this case, problem \eqref{svmseparable} is not feasible, and the following problem can be solved instead:
\begin{equation}\label{svmnonseparable}
\begin{array}{ll}
\displaystyle \min_{x} & \displaystyle\sum_{m=1}^M\max\left\{1-y_m\hat s_m^Tx,0\right\}^q+\frac{\rho}{2}\sum_{n=1}^{N-1}x_n^2,
\end{array}
\end{equation}where 
the constant $\rho\geq0$ balances the relative importance of minimizing the classification errors and maximizing the margin. Problem \eqref{svmnonseparable} with $q=1$ is called the soft-margin SVM in \cite{vapnik95svm}. It is clear that problem \eqref{svmnonseparable} is a special instance of \eqref{pnorm} with
$$A=\left[
           \begin{array}{c}
             y_1\hat s_1^T \\
             \vdots \\
             y_M\hat s_M^T  \\
           \end{array}
         \right],~b=e,~h(x)=\frac{\rho}{2}\sum_{n=1}^{N-1}x_n^2,~\text{and}~\X=\mathbb{R}^N.$$ Here, $e$ is the all-one vector of dimension $M.$

\noindent\textbf{Joint Power and Admission Control\cite{sidiropoulos11twcjpac,liu13tspjpac}.} Consider a wireless network consisting of $K$ interfering links (a link corresponds to a transmitter/receiver pair) with
channel gains $g_{kj}\geq0$ (from the transmitter of link $j$ to the
receiver of link $k$), noise power $\eta_k>0,$ signal-to-interference-plus-noise-ratio (SINR) target
$\gamma_k>0,$ and power budget $\bar p_k>0$ for $k, j=1,2,\ldots,K.$ {Denoting} the transmission power of transmitter $k$ by $x_k$, the SINR at the $k$-th receiver can be expressed as \begin{equation}\label{sinr}\displaystyle
\text{SINR}_k=\frac{g_{kk}x_k}{\eta_k+\displaystyle\sum_{j\neq
k}g_{kj}x_j},~k=1,2,\ldots,K.\end{equation} Due to {the existence} of mutual interferences among different links ({which correspond to} the term $\sum_{j\neq
k}g_{kj}x_j$ in \eqref{sinr}), the linear system $$\displaystyle \text{SINR}_k\geq \gamma_k,~\bar p_k\geq x_k\geq0,~k=1,2,\ldots,K$$ may not be feasible. The joint power and admission control problem aims at supporting a {maximum} number of links at their specified SINR targets while using a {minimum} total transmission power. {Assuming without loss of generality that} $g_{kk}=\gamma_k=\bar p_k=1$ for all $k=1,2,\dots,K,$ the joint power and admission control problem can be formulated as follows (see \cite{liu13tspjpac})
\begin{equation}\label{sparse-msp}
\begin{array}{ll}
\displaystyle \min_{x} & \left\|\max\left\{b-Ax,{0}\right\}\right\|_q^q+\rho e^Tx \\
\mbox{s.t.} & \displaystyle {0}\leq x\leq e,
\end{array}
\end{equation}
where $\rho>0$ is a parameter, $b=[\eta_1,\eta_2,\ldots,\eta_K]^T,$ and $A=[a_{kj}]\in\mathbb{R}^{K\times K}$ with
$$a_{kj}=\left\{\begin{array}{cl}
1,&\text{if~}k=j;\\
- g_{kj},\quad &\text{if~}k\neq
j.
\end{array}
\right.$$ By {utilizing} the special structure of $A,$ i.e., all of its diagonal entries are positive and nondiagonal entries are nonpositive, it is shown in \cite[Theorem 1]{liu13tspqnorm} that the solution of problem \eqref{sparse-msp} can maximize the number of supported links using a {minimum} total transmission power as long as $q$ is chosen to be sufficiently small (but not necessarily to be zero). Clearly, \eqref{sparse-msp} is a special case of \eqref{pnorm} with
$$M=K,~N=K,~h(x)=\rho e^Tx,~\text{and}~\X=\left\{x\,|\,{0}\leq x\leq e\right\}\subseteq \mathbb{R}^{N\times1}.$$

\noindent\textbf{Linear Decoding Problem\cite{candes05titdecoding}.} Given the coding matrix $C\in\mathbb{R}^{K_1\times K_2}$ and corrupted measurement $c=Cx+e_u\in\mathbb{R}^{K_1\times1},$ where $e_u$ is an unknown vector of errors, the linear decoding problem is to recover $x$ from $c.$ It is shown in \cite{candes05titdecoding} that, if $C$ satisfies the restricted isometry property, $x$ can be exactly recovered by solving the convex minimization problem $$\min_{x}\|c-Cx\|_1$$ provided that $e_u$ is sparse. By \cite[Theorem 4.10]{foucart2013mathematical}, $L_q$ ($q\in(0,1)$) minimization \begin{equation}\label{decoding}\min_{x}\|c-Cx\|_q^q\end{equation} has a better capability of recovering $x$ than $L_1$ minimization.
~By using the equation $|a|=\max\left\{a,0\right\}+\max\left\{-a,0\right\},$ it is simple to see problem \eqref{decoding} is a special case of \eqref{pnorm} with $$M=2K_1,~N=K_2,~A=\left[                                                                                                                 \begin{array}{c}
                                                                                                                                    C \\
                                                                                                                                    -C \\
                                                                                                                                  \end{array}
                                                                                                                                \right],~b=\left[
                                                                                                                                             \begin{array}{c}
                                                                                                                                               c \\
                                                                                                                                               -c \\
                                                                                                                                             \end{array}
                                                                                                                                           \right],~h(x)=0,~\text{and}~\X=\mathbb{R}^N.$$



\section*{Appendix B: Proof of Lemma \ref{lemma-local}}\label{appendix-lemma-local}
Let $\bar x$ be any local minimizer of problem \eqref{psub4} with $\I_{\bar x},\,\J_{\bar x},$ and $\K_{\bar x}$ given in \eqref{M1M2M3}.
For convenience, we denote $\I_{\bar x},\,\J_{\bar x},\,\K_{\bar x}$ as $\I,\,\J,\,\K$ in this proof.
We prove that $\bar x$ is a local minimizer of problem \eqref{pnorm}
by dividing the proof into two parts. The first one is the easy case where $\K=\emptyset$ and the second one deals with the complicated case where  $\K\neq\emptyset.$

\textbf{Part 1}: $\K=\emptyset.$ In this case, $\bar x$ is a local minimizer of problem \begin{equation*}
\begin{array}{cl}
\displaystyle \min_{x} & \displaystyle \|(b-Ax)_{{\J}}\|_q^q + h(x) \\
[5pt] \mbox{s.t.} & x\in \X.
\end{array}
\end{equation*} By the definition, $\bar x$ is a local minimizer of problem \eqref{pnorm}.

\textbf{Part 2}: $\K\neq\emptyset.$ Consider the feasible direction cone $\mathcal{D}_{\bar x}$ of problem \eqref{pnorm} at point $\bar x,$ i.e.,
 \begin{equation*}\label{feasibledcone}\D_{\bar x}=\left\{\,d\,|\,\bar x+\alpha d\in \X~\text{for~some}~\alpha>0\right\}.\end{equation*} For simplicity, we use $\mathcal{D}$ to denote $\mathcal{D}_{\bar x}$ in the subsequent proof.
  For any subset $\K_{p}$ of $\K$ indexed by $p=1,2,\ldots,P:=2^{|\K|},$ let $\K_{p}^{c}=\K\setminus\K_{p},$ and define
 \begin{equation*}
   \D_p=\left\{d\,|\,(Ad)_{\K_{p}}\leq 0,(Ad)_{\K_{p}^{c}}\geq 0\right\}\bigcap \D.
 \end{equation*}
 By the Minkowski-Weyl Theorem \cite[Proposition 3.2.1]{bertsekas03convex}, there exist $d_p^1, d_p^2, \ldots, d_p^{g_p} \in \D_p$, such that
 \begin{equation*}
   \D_p=\Cone\left\{d_p^{1},d_p^{2},\ldots,d_p^{g_p}\right\},
 \end{equation*} and thus
 \begin{equation*}
   \D=\displaystyle\bigcup_{p=1}^P\Cone\left\{d_p^{1},d_p^{2},\ldots,d_p^{g_p}\right\}.
 \end{equation*}Without loss of generality, assume $\|d_p^{j}\|=1$ for all $j=1,2,\ldots,g_p,~p=1,2,\ldots,P.$ For any $d\in\bigcup_{p=1}^P\left\{d_p^{1},d_p^{2},\ldots,d_p^{g_p}\right\}\subseteq \D$, define \begin{equation}\label{M2<}\overleftarrow{\K}^{d}=\left\{m\in \K\,|\,(Ad)_m<0\right\}.\end{equation} Next, we consider the two cases where $\overleftarrow{\K}^{d}$ is nonempty and empty, respectively. The former happens when $d$ is not a feasible direction of problem \eqref{psub4} at point $\bar x;$ while the latter happens when $d$ is a feasible direction of problem \eqref{psub4} at point $\bar x.$

 \textbf{Case 1:} $\overleftarrow{\K}^{d}\neq \emptyset.$
  Since $d\in \D,$ there must exist $\epsilon_0^d$ so that $\bar x+\epsilon d\in \X$ holds for all $0\leq \epsilon\leq\epsilon_0^d.$ Define \begin{equation}\label{M3>}\overrightarrow{\J}^{d}=\left\{m\in \J\,|\,(Ad)_m>0\right\}.\end{equation} Choose $\epsilon_1^d$ small enough such that
  \begin{equation}\label{smaller}\left(b-A\bar x-\epsilon_1^d Ad\right)_m\leq 0,~\forall~m\in {\I} \end{equation}
  and\begin{equation}\label{smaller'}\left(b-A\bar x-\epsilon_1^d Ad\right)_m\geq \frac{\left(b-A\bar x\right)_m}{2}>0,~\forall~m\in \overrightarrow{\J}^{d}.\end{equation}Therefore, for $0\leq\epsilon\leq\min\left\{\epsilon_0^d,\epsilon_1^d\right\}$, we obtain \begin{align}
                    &f(\bar x+\epsilon d)- f(\bar x)\nonumber\\
                    =~&\|\max\left\{b-A\left(\bar x+\epsilon d\right),0\right\}\|_q^q-\|\max\left\{b-A\bar x,0\right\}\|_q^q\nonumber\\
                    =~&\sum_{m\in \overleftarrow{\K}^{d}\cup\J} \left(b-A\bar x-\epsilon Ad\right)_m^q-\sum_{m\in\J}\left(b-A\bar x\right)_m^q\label{ineq1} \\
                    \geq~& \sum_{m\in\overleftarrow{\K}^{d}}(-Ad)_m^q\epsilon^q + \sum_{m\in {\overrightarrow{\J}^{d}}} \left(\left(b-A\bar x-\epsilon Ad\right)_m^q-\left(b-A\bar x\right)_m^q\right)\label{ineq2}\\
                    \geq~& \sum_{m\in\overleftarrow{\K}^{d}}(-Ad)_m^q\epsilon^q + \sum_{m\in {\overrightarrow{\J}^{d}}} q \left(b-A\bar x-\epsilon Ad\right)_m^{q-1}\left(-\epsilon(Ad)_m\right)\label{concave}\\
                    \geq~& \sum_{m\in\overleftarrow{\K}^{d}}(-Ad)_m^q\epsilon^q + \sum_{m\in {\overrightarrow{\J}^{d}}} q \left(\frac{(b-A\bar x)_m}{2}\right)^{q-1}\left(-\epsilon(Ad)_m\right),\label{smaller2}
  \end{align}
  where \eqref{ineq1} is due to \eqref{M1M2M3}, \eqref{M2<}, and \eqref{smaller}\bl{;} \eqref{ineq2} is due to \eqref{M3>}; \eqref{concave} is due to the concavity of the function $z^q$ with respect to $z>0;$~\eqref{smaller2} is due to \eqref{smaller'} and the definition of $\overrightarrow{\J}^{d}$ in \eqref{M3>}. Moreover, by \eqref{L} and the Taylor's expansion, for any $0\leq\epsilon\leq 1$, there exists $\xi\in(0,1)$ such that
  \begin{equation}\label{nablah}
   \begin{array}{rl}
      h(\bar x+\epsilon d)-h(\bar x)=&\epsilon\nabla h(\bar x+\xi \epsilon d)^Td\\
                                     \geq &-\epsilon \left\|\nabla h(\bar x+\xi \epsilon d)\right\|\\
                                     \geq &-\epsilon\left(\left\|\nabla h(\bar x)\right\|+\epsilon L_h\right)\\
      \geq & -\epsilon\left(\left\|\nabla h(\bar x)\right\|+L_h\right).
   \end{array}
  \end{equation}
  Combining \eqref{smaller2} with \eqref{nablah}, for any $0\leq\epsilon\leq\min\left\{\epsilon_0^d,\epsilon_1^d,1\right\},$ we obtain
  $$F(\bar x+\epsilon d)- F(\bar x)\geq \lambda_1^d\epsilon^q-\lambda_2^d\epsilon,$$ where
  \begin{equation*}\label{lambda1}\lambda_1^d:=\sum_{m\in\overleftarrow{\K}^{d}}(-Ad)_m^q\epsilon^q>0,\end{equation*}\begin{equation*}\label{lambda2}\lambda_2^d:=\sum_{m\in \overrightarrow{\J}^{d}} q \left(\frac{(b-A\bar x)_m}{2}\right)^{q-1}(Ad)_m+\left\|\nabla h(\bar x)\right\|+L_h>0.\end{equation*}
  Define \begin{equation*}\label{epsilon}\epsilon_2^d:=\left(\frac{\lambda_1^d}{\lambda_2^d}\right)^{\frac{1}{1-q}}\end{equation*} and
  \begin{equation*}\label{epsilond}
    \bar \epsilon^d:=
\min\left\{\epsilon_0^d,\epsilon_1^d,\epsilon_2^d,1\right\}>0.
  \end{equation*}
From the above analysis, we can conclude that, for any $d\in\bigcup_{p=1}^P\left\{d_p^{1},d_p^{2},\ldots,d_p^{g_p}\right\}$ with $\overleftarrow{\K}^{d}\neq \emptyset,$ {$F(\bar x+\epsilon d)\geq F(\bar x)$ holds for all $\epsilon\in[0,\bar\epsilon^d].$}

\textbf{Case 2}: $\overleftarrow{\K}^{d}=\emptyset.$ Recall the definition of $\overleftarrow{\K}^{d}$ (cf. \eqref{M2<}). $\overleftarrow{\K}^{d}=\emptyset$ implies that $d$ is a feasible direction of problem \eqref{psub4} at point $\bar x.$ From the assumption that $\bar x$ is a local minimizer of problem \eqref{psub4}, we know that there exists an $\tilde \epsilon>0$ such that for all $d\in \bigcup_{p=1}^P\left\{d_p^{1},d_p^{2},\ldots,d_p^{g_p}\right\}$ with $\overleftarrow{\K}^{d}=\emptyset,$ there holds $F(\bar x+\epsilon d)\geq F(\bar x)$ for all $\epsilon\in[0,\tilde\epsilon].$

We now combine the above two cases: \textbf{Case 1} and \textbf{Case 2}. Since there are finitely many directions $\bigcup_{p=1}^P\left\{d_p^{1},d_p^{2},\ldots,d_p^{g_p}\right\},$ it follows that $$\bar\epsilon:=\min\left\{\min_{\overleftarrow{\K}^{d}\neq \emptyset,\,j=1,\ldots,g_p,\,p=1,\ldots,P}\left\{\bar \epsilon^{d_p^j}\right\},\tilde\epsilon\right\}>0$$ and 
\begin{equation}\label{decrease}\bar x+\epsilon d_p^j\in \X,~F(\bar x+\epsilon d_p^j)\geq F(\bar x),~\forall~j=1,2,\ldots,g_p,\,p=1,2,\ldots,P\end{equation} hold true for all $\epsilon\in[0, \bar \epsilon].$

  Let $\Conv_p(\bar x,\bar \epsilon)$ denote the convex hull spanned by points $\bar x$ and $\bar x+\bar \epsilon d_p^{j},~j=1,2,\ldots,g_p.$ Then, for any $x\in\bigcup_{p=1}^P\Conv_p(\bar x,\bar \epsilon),$ we have $F(x)\geq F(\bar x)$ by \eqref{decrease} and the fact that $F(x)$ is concave in $\Conv_p(\bar x,\bar \epsilon)$. 
  Furthermore, one can always choose a sufficiently small but fixed $\epsilon>0$ such that $B(\bar x,\epsilon)\bigcap \X \subseteq \bigcup_{p=1}^P\Conv_p(\bar x,\bar \epsilon).$ Therefore, $\bar x$ is a local minimizer of problem \eqref{pnorm}.\qed

\section*{Appendix C: Proof of Lemma \ref{lemma-smooth}}\label{appendix-lemma-smooth-approximation}

We show the three items of Lemma \ref{lemma-smooth} separately.

  \noindent (i) of Lemma \ref{lemma-smooth}: it follows directly from the inequality $$\theta^q(t)\leq \theta^q(t,\mu)\leq\left(\theta(t)+\frac{\mu}{2}\right)^q\leq \theta^q(t)+\left(\frac{\mu}{2}\right)^q,~\forall~t\leq \mu.$$

\noindent (ii) of Lemma \ref{lemma-smooth}: When $t\neq 0$ and $t\neq\mu,$ $\theta^q(t,\mu)$ is twice continuously differentiable with respect to $t$. Recall $\theta^q(t,\mu)\geq \left(\mu/2\right)^q$ for all $t$ (cf. \eqref{lowerq}). Then it follows from \eqref{q2nd} that
$$\left|\left[\theta^q(t,\mu)\right]''\right|\leq 4q\mu^{q-2},~\forall~t\notin\left\{0,\mu\right\}.$$ This further implies (ii) of Lemma \ref{lemma-smooth}.

\noindent (iii) of Lemma \ref{lemma-smooth}: By the mean-value theorem \cite[Theorem 2.3.7]{clarke}, we have
    \begin{equation}\label{taylor}\theta^q(t,\mu)= \theta^q(\hat t,\mu)+\left[\theta^q(\hat t,\mu)\right]'\left(t-\hat t\right)+\frac{\upsilon}{2}\left(t-\hat t\right)^2,\end{equation} where $\upsilon\in\partial_{t}\left(\left[\theta^q(\xi\hat t+(1-\xi)t,\mu)\right]'\right)$ and $\xi\in[0,1].$ We consider the following three cases.
\begin{itemize}
  \item Case $\hat t>2\mu:$ Since $t-\hat t\geq {-\hat t}/{2},$ {it follows for any $\xi\in[0,1]$ that}  
$$\xi t+(1-\xi)\hat t=\hat t+\xi(t-\hat t)\geq \hat t/2>\mu.$$ This, together with \eqref{q2nd} and \eqref{kappa}, implies {that the} $\upsilon$ in \eqref{taylor} satisfies $$\upsilon\leq 0= \kappa(\hat t,\mu).$$ From this and \eqref{taylor}, we obtain \eqref{q-upperbnd}.
\item Case $\hat t\in[-\mu, 2\mu]:$ 
From (ii) of Lemma \ref{lemma-smooth}, $|\upsilon|$ is uniformly bounded by $\kappa(\hat t,\mu)={{4q}\mu^{q-2}}.$ Combining this with \eqref{taylor} yields \eqref{q-upperbnd}. 
\item Case $\hat t<-\mu:$ Since $t-\hat t\leq \mu,$ for any $\xi\in[0,1],$ it follows
$$\xi t+(1-\xi)\hat t=\hat t+\xi(t-\hat t)< 0.$$ From this, \eqref{q2nd}, \eqref{kappa}, and \eqref{taylor}, we can obtain \eqref{q-upperbnd}.
\end{itemize}  This completes the proof of Lemma \ref{lemma-smooth}. \qed

\begin{acknowledgements}
We would like to thank Prof. Xiaojun Chen and Dr. Wei Bian for many insightful comments, which helped us in improving the results in this paper. We also thank Dr. Qingna Li and Dr. Xin Liu for many useful discussions on an early version of this paper.
\end{acknowledgements}



\bibliography{liubibfiles}
\bibliographystyle{spmpsci}

\end{document}